\documentclass[10pt]{article}
\usepackage{listings}
\usepackage{pdflscape}
\usepackage{amsfonts}
\usepackage{epsfig}
\usepackage{lscape}
\usepackage{longtable}
\usepackage{amsmath,amssymb,amsthm}
\usepackage{graphicx}
\usepackage{subfigure}
\usepackage{color}
\usepackage{placeins}
\usepackage{url}
\usepackage{cases}
\usepackage{longtable}
\usepackage{supertabular}
\usepackage{multirow}

\usepackage{amsmath}
\allowdisplaybreaks[4]
\usepackage{cite}
\usepackage{algorithmic}
\usepackage{algorithm,float}
\usepackage{booktabs}
\usepackage[shortlabels]{enumitem}
\usepackage{hyperref}
\usepackage[online]{threeparttablex}
\usepackage{siunitx}
\usepackage{tikz}
\usepackage{bm}
\usepackage{arydshln}
\hypersetup{hidelinks}


\newtheorem{theorem}{Theorem}
\newtheorem{assumption}{Assumption}

\newtheorem{remark}{Remark}
\newtheorem{definition}{Definition}
\newtheorem{proposition}{Proposition}
\newtheorem{corollary}{Corollary}

\begin{document}
\title{\bf \Large{Robust principal component analysis with rank and cardinality regularization under matrix factorization}\footnotemark[1]}
\author{Wenjing Li\footnotemark[2], \quad Wei Bian \footnotemark[3], \quad Kim-Chuan Toh\footnotemark[4]}

\date{}
\maketitle

\renewcommand{\thefootnote}{\fnsymbol{footnote}}
\footnotetext[1]{This work is funded by the National Natural Science Foundation of China Grants (12425115, 12271127, 12301397, 62176073).}
\footnotetext[2]{School of Mathematics, Harbin Institute of Technology, Harbin 150001, China. Part of this work is done while the author is a visiting scholar in the Department of Mathematics at the National University of Singapore (\tt{liwenjingsx@163.com}).}
\footnotetext[3]{Corresponding author. School of Mathematics, Harbin Institute of Technology, Harbin 150001, China (\tt{bianweilvse520@ 163.com}).}
\footnotetext[4]{Department of Mathematics and Institute of Operations Research and Analytics, National University of Singapore, 10 Lower Kent Ridge Road, Singapore (\tt{mattohkc@nus.edu.sg}).}
\renewcommand{\thefootnote}{\arabic{footnote}}

\begin{abstract}
Robust principal component analysis is an important representative method in data analysis. It is usually viewed as an optimization problem involving the rank and $\ell_0$-norm of matrices. In this paper, we study the rank and $\ell_0$ regularized optimization problem and its matrix factorization problem. We establish their equivalences on global minimizers and stationary points, respectively. Furthermore, we construct a broadly applicable equivalent nonconvex relaxation framework for the constrained factorization model in the sense of global minimizers and stationary points with strong optimality conditions (called strong stationary points). For the general factorization problem with lower semicontinuous regularizers and a loss function whose gradient is locally Lipschitz, we propose a novel proximal gradient-based algorithm based on joint and alternating calculation with convergence to its limiting-critical points. The algorithm can attain the stationary points of the original problem and its adaptive counterpart can attain the strong stationary points of the factorization problem.
\end{abstract}


\section{Introduction}
Robust principal component analysis (RPCA) is one of the most representative approaches for data analysis, that is often used for image denoising and video processing. It has been widely viewed as the task of recovering a low-rank matrix from sparsely corrupted data. Denote $\varOmega\subseteq\{1,\ldots,m\}\times\{1,\ldots,n\}$ as the index set of observed entries, and $\mathcal{P}_{\varOmega}:\mathbb{R}^{m\times n}\rightarrow \mathbb{R}^{m\times n}$ as the operator such that for any $Z\in\mathbb{R}^{m\times n}$, $(\mathcal{P}_{\varOmega}(Z))_{ij}=Z_{ij}\ \text{if}\ (i,j)\in\varOmega\ \text{and}\ 0\ \text{otherwise}$. RPCA asks to recover the low-rank matrix $Z^l\in\mathbb{R}^{m\times n}$ from the observed matrix $\mathcal{P}_{\varOmega}(Z^s)$ with $Z^s=Z^l+S^u$, where $S^u\in\mathbb{R}^{m\times n}$ is an unknown sparse corruption matrix. For any matrix $S\in\mathbb{R}^{m\times n}$, its sparsity is usually measured by the $\ell_0$-norm, defined as $\|S\|_0=\sum_{i=1}^m\sum_{j=1}^n\vartheta(|S_{ij}|)$, where \begin{equation}\label{c_def}
	\text{$\vartheta:\mathbb{R}_+\rightarrow\mathbb{R}_+$ satisfies that $\vartheta(t)=1$ if $t>0$, and $\vartheta(t)=0$ otherwise.}
\end{equation} The ideal optimization model of RPCA is as follows,
\begin{equation}\label{ori-m}
	\min_{Z,S\in\mathbb{R}^{m\times n}}~\text{rank}(Z)+\alpha\,\|S\|_0,\quad\text{s.t. }\mathcal{P}_{\varOmega}(Z+S)=\mathcal{P}_{\varOmega}(Z^s)
\end{equation}
with parameter $\alpha>0$, in which $S$ represents the sparse matrix variable. Due to the nature of rank function and $\ell_0$-norm, this problem is NP-hard in general. Considering the various noise types in real-world applications, the actual observation is $\mathcal{P}_{\varOmega}(Z^l+S^u+R^u)$, where $Z^l$ is the low-rank target matrix, and $S^u$ and $R^u$ are unknown sparse corruption matrix and random noise matrix, respectively. This can be regarded as the stable RPCA formulation in \cite{bouwmans2018applications}. As the target matrix $Z^l$ is inherently low-rank, there is typically an anticipated upper bound $d$ on its rank, i.e., $\text{rank}(Z^l)\leq d$. Without prior knowledge about the noise $R^u$ in practice, the RPCA model can be reasonably expressed by the rank and $\ell_0$ regularized problem,
\begin{equation}\label{m_s_l}
	\tag{L}
	\min_{(Z,S)\in\mathbb{R}^{m\times n}\times\mathcal{B}^S:\,\text{rank}(Z)\leq d}~F_0(Z,S):=f(Z,S)+2\lambda\,\text{rank}(Z)+\beta\|S\|_0,\\
\end{equation}
where $f:\mathbb{R}^{m\times n}\times\mathbb{R}^{m\times n}\rightarrow\mathbb{R}_+$ is locally Lipschitz continuous, $\mathcal{B}^S:=[\bm{\kappa}^1,\bm{\kappa}^2]\subseteq\mathbb{R}^{m\times n}$ with ${\bm{\kappa}}^1_{ij}\in[-\infty,0]$ and ${\bm{\kappa}}^2_{ij}\in[0,+\infty],\forall i,j$, and parameters $\lambda$ and $\beta$ are nonnegative. Problem (\ref{m_s_l}) is a novel model with theoretical challenges and computational potential for RPCA. For optimization problems involving the rank function, a common approach is to solve the corresponding relaxation problems with  nuclear norm \cite{Srebro2005}, weighted nuclear norm \cite{Gu_2014_CVPR}, truncated nuclear norm \cite{HuPAMI2013} or Schatten-$p$ ($0<p<1$) norm \cite{ShangAAAI2016}, and its calculations usually involve a singular value decomposition (SVD) in each iteration \cite{bouwmans2018applications,Cai2010SIOPT,yao2019large}. This method often incur a high computational cost and is not appropriate for processing high-dimensional data.
\medskip

For any matrix $Z\in\mathbb{R}^{m\times n}$ with $\text{rank}(Z)\leq d$, it can be formulated as $Z=XY^{\mathbb{T}}$ with two reduced-size matrices $X\in\mathbb{R}^{m\times d}$ and $Y\in\mathbb{R}^{n\times d}$. This matrix factorization is an effective dimensionality reduction method for optimization problems with low-rank matrices as variables. As shown in \cite{Fan2019factor}, when $\text{rank}(Z)\leq d$, the rank and column sparsity of matrices have the equivalent relation that $\text{rank}(Z)=\min_{(X,Y)\in\mathbb{R}^{m\times d}\times\mathbb{R}^{n\times d}:XY^{\mathbb{T}}=Z}\frac{1}{2}\big(\text{nnzc}(X)$ $+\text{nnzc}(Y)\big)$,
where $\text{nnzc}(\cdotp)$ denotes the number of non-zero columns of the matrix. In light of these, it is not hard to deduce the equivalent factorization problem (\ref{s_l_nc}) of problem (\ref{m_s_l}) in the sense of global minimizers as follows:
\begin{equation}\label{s_l_nc}
	\tag{F}
	\min_{(X,Y,S)\in\mathbb{R}^{m\times d}\times\mathbb{R}^{n\times d}\times\mathcal{B}^S}~F(X,Y,S):=f(XY^{\mathbb{T}},S)+\lambda\big(\text{nnzc}(X)+\text{nnzc}(Y)\big)+\beta\|S\|_0.
\end{equation}
In the expression of function $F$, the terms are loss function, low-rank inducing regularizer and sparse regularizer in order. The matrix decomposition can reduce the variable dimension of (\ref{m_s_l}) and avoid the calculation of SVD, so that problem (\ref{s_l_nc}) has significant computational advantages. By incorporating practical prior information as constraints, we can focus on the factorization problem (\ref{s_l_nc}) with constraints
\begin{equation}\label{s_l}
	\tag{FB}
	\min_{(X,Y,S)\in\mathcal{B}^X\times \mathcal{B}^Y\times \mathcal{B}^S}~F(X,Y,S),
\end{equation}
where $\mathcal{B}^X:=\{X\in\mathbb{R}^{m\times d}:\|X_i\|_2\leq\tau,\,\forall i\}$ and $\mathcal{B}^Y:=\{Y\in\mathbb{R}^{n\times d}:\|Y_i\|_2\leq\tau,\,\forall i\}$, with $X_i$ and $Y_i$ being the $i$-th column of $X$ and $Y$, respectively, and $\tau\in[0,+\infty]$ being a given bound. Model (\ref{s_l}) is novel for RPCA, and encompasses many models with broad application backgrounds. When $\tau=+\infty$, problem (\ref{s_l}) is problem (\ref{s_l_nc}). When $\lambda=0$, problem (\ref{s_l}) is an $\ell_0$ regularized problem, and when $\beta=0$, problem (\ref{s_l}) is a rank regularized factorization problem. These models are NP-hard to solve in general but are popular ideal models in high-dimensional data processing, such as wireless communication, image restoration, system identification and collaborative filtering. In \cite{Li2024}, problem (\ref{s_l}) with $\beta=0$ and $f(XY^{\mathbb{T}},S)$ independent of $S$, is studied together with its capped-$\ell_1$ relaxation problem. In \cite{zhao2022robust}, with bounded constraints, the robust tensor completion model with tensor average rank and $\ell_0$-norm is considered, based on the equivalence to some nonconvex relaxation problems in the sense of global minimizers.
\medskip

Problems (\ref{m_s_l}) and (\ref{s_l_nc}) with $\beta=0$ and $f(XY^{\mathbb{T}},S)$ independent of $S$ are studied in \cite{Li2024}, proving that any local minimizer of (\ref{m_s_l}) is a local minimizer of (\ref{s_l_nc}), and establishing an inclusion relation on the global minimizers among (\ref{m_s_l}), (\ref{s_l_nc}) and (\ref{s_l}) where appropriate. In this paper, we will conduct an in-depth analysis on the relations of global minimizers, local minimizers, and stationary points among (\ref{m_s_l}), (\ref{s_l_nc}) and (\ref{s_l}) in the general case. By proving the equivalence between (\ref{m_s_l}) and (\ref{s_l_nc}) in the sense of stationary points, we can obtain the stationary points of (\ref{m_s_l}) by solving (\ref{s_l_nc}) in theory. In particular, the equivalence on the stationary points is also the latest conclusion for the problem $\min_{Z:\,\text{rank}(Z)\leq d}\{f(Z)+\lambda\,\text{rank}(Z)\}$ considered in \cite{Li2024}.
\medskip

The function $\text{nnzc}(\cdotp)$ can be equivalently regarded as the $\ell_{p,0}(p\geq0)$-norm, since $\text{nnzc}(Z)=\|(\|Z_1\|_p,\cdots,$ $\|Z_n\|_p)\|_0,\forall Z\in\mathbb{R}^{m\times n}$, where $Z_i$ denotes the $i$-th column of $Z$ and $\|\cdotp\|_p$ is the vector $\ell_p$-norm. It follows that problem (\ref{s_l}) belongs to mixed group $\ell_{p,0}$ regularized problems with nonconvex loss. For the $\ell_0$($\ell_{p,0}$) regularized problems, a popular method is to solve their relaxation problems. Common relaxations functions of the $\ell_0$-norm include $\ell_1$-norm, $\ell_p$-norm ($0<p<1$) \cite{Foucart2009}, capped-$\ell_1$ penalty \cite{Peleg2008}, smoothly clipped absolute deviation (SCAD) penalty \cite{Fan2001} and minimax concave penalty (MCP) \cite{Zhang2010}, and they can be extended to the $\ell_{p,0}$-norm. The models with some specific nonconvex relaxations are equivalent to the corresponding $\ell_0$($\ell_{p,0}$) regularized problems in the sense of global minimizers \cite{Bian2020,LeThi2015,Li2022,Soubies2017}. In particular, with the help of the specific capped-$\ell_1$ relaxation, a class of relaxation models constructed in \cite{LeThi2015} are equivalent to the $\ell_0$ regularized problem. Moreover, some equivalence on optimality conditions beyond global minimizers between (\ref{s_l}) with $\lambda=0$ and $f(XY^{\mathbb{T}},S)$ independent of $XY^{\mathbb{T}}$, or $\beta=0$ and $f(XY^{\mathbb{T}},S)$ independent of $S$, and its capped-$\ell_1$ relaxation was established in \cite{Bian2020,Li2024}. In this paper, we will construct an equivalent relaxation framework for (\ref{s_l}) in the sense of global minimizers and some optimality conditions, which contains diverse relaxations. The general relaxation theory is also new for the $\ell_0$ regularized problem involving nonconvex loss function.
\medskip

There exist some algorithms to solve the unconstrained version of (\ref{m_s_l}) and its relaxations in \cite{wen2020robust,wen2019nonconvex}, and relaxations of (\ref{ori-m}) and (\ref{s_l}) in \cite{giampouras2018robust,guo2022rank,shang2018bilinear}, with smooth losses. Moreover, some existing algorithms in \cite{latafat2022block,SabachSIIMS2016,YinJSC2017} can solve (\ref{s_l}) with a smooth loss. However, they can only achieve critical points under certain conditions in theory, which may be far away from global minimizers of (\ref{s_l}). Recently, two classes of adaptive alternating proximal gradient (AAPG) algorithms were proposed to solve (\ref{s_l}) with $\beta=0$ and $f(XY^{\mathbb{T}},S)$ independent of $S$. For the general case, it is meaningful to strengthen the optimality condition of stationary points, design an efficient algorithm with guaranteed convergent to strengthened stationary points and build upon AAPG algorithms to make some improvements for (\ref{s_l}). In the algorithmic part, we will consider the following general nonconvex (not necessarily continuous) factorization optimization problem,
\begin{equation}\label{GFP}
	\tag{GF}
	\min_{(X,Y,S)\in\mathbb{R}^{m\times d}\times\mathbb{R}^{n\times d}\times\mathbb{R}^{m\times n}}~G(X,Y,S):=f(XY^{\mathbb{T}},S)+\lambda\big(u(X)+v(Y)\big)+\beta w(S),
\end{equation}
where $f:\mathbb{R}^{m\times n}\times\mathbb{R}^{m\times n}\rightarrow\mathbb{R}$ is differentiable and its gradient $\nabla f$ is locally Lipschitz continuous, parameters $\lambda$ and $\beta$ are nonnegative, $u:\mathbb{R}^{m\times d}\rightarrow\overline{\mathbb{R}}:=(-\infty,+\infty]$, $v:\mathbb{R}^{n\times d}\rightarrow\overline{\mathbb{R}}$ and $w:\mathbb{R}^{m\times n}\rightarrow\overline{\mathbb{R}}$ are proper, lower semicontinuous and bounded from below by affine functions, and their proximal operators can be calculated. Considering the locally Lipschitz continuity of $\nabla f$ and disregarding the structure of $f(XY^{\mathbb{T}},S)$, the existing proximal gradient algorithms in \cite{jia2023convergence,kanzow2025convergence,kanzow2022convergence} can solve (\ref{GFP}). Taking account of the loss function structure, we will propose a proximal gradient-based algorithm to solve (\ref{GFP}) by employing both joint and alternating calculation. Furthermore, we will show that the designed algorithm can attain the stationary points of (\ref{m_s_l}), and its adaptive counterpart can converge to the strengthened stationary points of (\ref{s_l}) and its relaxations.
\medskip

In this paper, we focus on establishing the solution relations between the nonconvex problems (\ref{m_s_l}) and (\ref{s_l_nc}), constructing the continuous nonconvex relaxation framework for problem (\ref{s_l}), solving the general factorization problem (\ref{GFP}), including problem (\ref{s_l}) and its continuous relaxations as specific cases, and improving the theoretical results in some specific cases. The main contributions are as follows.

\begin{itemize}
	\item[(i)] Analyze the equivalent optimality conditions to local minimizers of problems (\ref{m_s_l}) and (\ref{s_l}). Prove the equivalences on global minimizers and stationary points between problems (\ref{m_s_l}) and (\ref{s_l_nc}). This establishes a direct relation from finding the stationary points of (\ref{s_l_nc}) to the stationary points of (\ref{m_s_l}).
	\item[(ii)] Delve into the optimality conditions of problem (\ref{s_l}), which are not satisfied by some non-global local minimizers. In view of them, define a class of stationary points with stronger optimality conditions (called strong stationary points) than the existing stationary point for (\ref{s_l}). In the sense of global minimizers and defined strong stationary points, construct a general equivalent nonconvex relaxation framework of (\ref{s_l}).
	\item[(iii)] Consider the general factorization problem (\ref{GFP}), which includes problem (\ref{s_l}) and its equivalent relaxation problems as specific cases. Propose a proximal gradient-based algorithm with joint and alternating calculation (JA-PG) that converges to the limiting-critical points of (\ref{GFP}). Combining with the equivalence between (\ref{m_s_l}) and (\ref{s_l_nc}) in the sense of stationary points, we also achieve the convergence of JA-PG to the stationary points for (\ref{m_s_l}).
	\item[(iv)] Analyze the computational commonalities of proximal operators associated with the relaxation functions. Based on the derived optimality conditions, develop an adaptive counterpart of JA-PG algorithm to solve the relaxation problems of (\ref{s_l}). Leveraging the equivalent relaxation theory, prove the convergence of the proposed algorithm to strong stationary points of (\ref{s_l}), which also establishes the algorithmic applicability of the relaxation framework.
\end{itemize}
\medskip

The rest of this paper is organized as follows. Section \ref{section 2} shows some properties of (\ref{m_s_l}), (\ref{s_l_nc}) and (\ref{s_l}), and relations in the sense of their global minimizers, local minimizers and stationary points. Section \ref{section 3} provides a general framework for the equivalent nonconvex relaxations of (\ref{s_l}) and present some verifiable conditions to support the equivalences on global minimizers and strong stationary points between (\ref{s_l}) and its relaxations. Section \ref{section 4} presents a proximal gradient-based algorithms with joint and alternating calculation to solve the general factorization problem and shows its adaptive counterpart to solve (\ref{s_l}) and its relaxations with improved convergence results. Section \ref{section 5} provides some numerical experiments to show the efficiency of the proposed models and algorithms. Finally, Section \ref{section 6} gives the conclusions.
\medskip

\textbf{Notations:}
Let $\mathcal{N}_{\infty}^{\sharp}$ be the set of infinite subsets of the set of natural numbers $\mathbb{N}$. The set of all positive integers is denoted by $\mathbb{N}_+$. For any $n\in\mathbb{N}_+$, denote $[n]:=\{1,2,\ldots,n\}$. Denote $\mathbb{R}_+:=[0,+\infty)$ and $\overline{\mathbb{R}}:=(-\infty,+\infty]$. Let $\mathbb{R}^{m\times n}$ be the space of $m\times n$ real matrices endowed with the inner product $\langle X,Y\rangle=\text{trace}(X^{\mathbb{T}}Y)$ and its induced Frobenius norm $\|\cdot\|_F$. For any scalar $\varepsilon>0$, set $\Omega\subseteq\mathbb{R}^{m\times n}$, $\mathbb{I}\subseteq[n]$, $\mathbb{J}\subseteq[m]\times[n]$, and matrix $U\in\mathbb{R}^{m\times n}$, define $B_\varepsilon(U):=\{Z\in\mathbb{R}^{m\times n}:\|Z-U\|_F\leq\varepsilon\}$, $B_\varepsilon:=B_\varepsilon(\bm{0})$, denote $U_i$ as the $i$-th column of $U$, define $U_{\mathbb{I}}:=[U_{i_1},\ldots,U_{i_{|\mathbb{I}|}}]\in\mathbb{R}^{m\times|\mathbb{I}|}$ with the cardinality $|\mathbb{I}|$ of $\mathbb{I}$, $i_j\in\mathbb{I},\forall j\in[|\mathbb{I}|]$ and $i_1<\cdots<i_{|\mathbb{I}|}$, and $\Omega_{\mathbb{I}}:=\{U_{\mathbb{I}}:U\in\Omega\}$, the restricted matrices $\underline{U}_{\mathbb{I}}$, $\underline{U}_{\mathbb{J}}\in\mathbb{R}^{m\times n}$ and sets ${\underline{\Omega}}_{\mathbb{I}}$, ${\underline{\Omega}}_{\mathbb{J}}\subseteq\mathbb{R}^{m\times n}$ are defined as 
$$(\underline{U}_{\mathbb{I}})_{i}=U_{i},\forall i\in\mathbb{I},\,(\underline{U}_{\mathbb{I}})_{i}=\bm{0},\forall i\notin\mathbb{I},\text{ and }{\underline{\Omega}}_{\mathbb{I}}:=\{\underline{U}_{\mathbb{I}}:U\in\Omega\},$$ $$(\underline{U}_{\mathbb{J}})_{ij}=U_{ij},\forall (i,j)\in\mathbb{J},\,(\underline{U}_{\mathbb{J}})_{ij}=0,\forall (i,j)\notin\mathbb{J},\text{ and }{\underline{\Omega}}_{\mathbb{J}}:=\{\underline{U}_{\mathbb{J}}:U\in\Omega\},$$
the indicator sets $\mathbb{I}_U\subseteq[n]$ and $\mathbb{J}_U\subseteq[m]\times[n]$ are defined as 
$$\mathbb{I}_U:=\{i\in[n]:U_i\neq\bm{0}\}\text{ if }\lambda>0\text{ and }[n]\text{ if }\lambda=0,$$ $$\mathbb{J}_U:=\{(i,j)\in[m]\times[n]:U_{ij}\neq0\}\text{ if }\beta>0\text{ and }[m]\times[n]\text{ if }\beta=0,$$
define $\mathbb{I}^c:=[n]\setminus\mathbb{I}$ and $\mathbb{J}^c:=[m]\times[n]\setminus\mathbb{J}$. For a matrix set $\mathcal{M}\subseteq\mathbb{R}^{m\times n}$, $\|\mathcal{M}\|_{\max}:=\sup\{\|U\|_2:U\in\mathcal{M}\}$. For a nonempty set $\Omega\subseteq\mathbb{R}^{m\times n}$ and a matrix $\bar{U}\in\Omega$, $N_{\Omega}(\bar{U})$ denotes the normal cone to $\Omega$ at $\bar{U}$ defined as in \cite{Rockafellar1998}, the interior of $\Omega$ is defined by $\text{int}(\Omega)$ and the indicator function $\delta_\Omega$ is defined by $\delta_\Omega(U)=0$ if $U\in\Omega$ and $\delta_\Omega(U)=+\infty$ otherwise. For any locally Lipschitz continuous function $f:\mathbb{R}^{m\times n}\times\mathbb{R}^{m\times n}\rightarrow\mathbb{R}$, denote $\partial f(\cdot)$ the limiting subdifferential \cite{Rockafellar1998} and $\partial^cf(\cdot)$ the Clarke subdifferential \cite{Clarke1983}. For a proper lower semicontinuous function $h:\mathbb{R}^{m\times n}\rightarrow\overline{\mathbb{R}}$, its domain is defined as $\text{dom}(h):=\{U\in\mathbb{R}^{m\times n}:h(U)<+\infty\}$, and its proximal operator is defined by ${\rm{prox}}_h(Z):={\rm{argmin}}_{V\in\mathbb{R}^{m\times n}}\{h(V)+\frac{1}{2}\|V-Z\|^2_F\}$. For any $Z\in\mathbb{R}^{m\times n}$, denote $d_Z:=\text{rank}(Z)$, $U_Z\in\mathbb{R}^{m\times d_Z}$ and $V_Z\in\mathbb{R}^{n\times d_Z}$ as orthogonal matrices from the compact SVD of $Z$, i.e. $Z=U_Z\Sigma_Z V_Z^{\mathbb{T}}$ with $\Sigma_Z\in\mathbb{R}^{d_Z\times{d_Z}}$ being a diagonal matrix with diagonal elements being the nonzero singular values of $Z$ in descending order, and $\Sigma_Z^{1/2}\in\mathbb{R}^{d_Z\times{d_Z}}$ as the matrix with elements corresponding to square roots of those of $\Sigma_Z$. Define $I_{+}:\mathbb{R}_+\rightarrow(0,+\infty]$ by $I_{+}(t)=t$ if $t>0$ and $I_{+}(t)=+\infty$ otherwise. For any problem ($\vartriangle$), denote the optimal value and the sets of global minimizers and local minimizers as $V_{(\vartriangle)}$, $\mathcal{G}_{(\vartriangle)}$ and $\mathcal{L}_{(\vartriangle)}$, respectively. We use \emph{iff} to denote ``if and only if''.

\section{Equivalences and properties of problems (\ref{m_s_l}), (\ref{s_l_nc}) and (\ref{s_l})}\label{section 2}
In this section, we analyze the relations among problems (\ref{s_l_nc}), (\ref{m_s_l}), (\ref{s_l}), and (\ref{m_s_l}) with constraints as follows,
\begin{equation}\label{m_s_l_b}
	\tag{LB}
	\min_{(Z,S)\in \mathcal{B}^Z\times \mathcal{B}^S:\,\text{rank}(Z)\leq d}~F_0(Z,S),
\end{equation}
where $\mathcal{B}^Z\subseteq\mathbb{R}^{m\times n}$ is a box constraint set. We analyze the properties of problems (\ref{m_s_l_b}) and (\ref{s_l}), which facilitate the subsequent introduction of strong stationary points for problem (\ref{s_l}). Furthermore, we establish the equivalences on global minimizers and stationary points between problems (\ref{m_s_l}) and (\ref{s_l_nc}). The equivalence in the sense of stationary points establishes a direct link between the solutions of these problems.

First, we show the equivalence on the global minimizers between (\ref{m_s_l}) and (\ref{s_l_nc}).
\begin{proposition}\label{gls_mf}
	Problems (\ref{m_s_l}) and (\ref{s_l_nc}) own the same optimal function value and their global minimizers can be transformed into each other, i.e. 
	\begin{equation}\label{lf_gl_transf}
		V_{(\ref{m_s_l})}=V_{(\ref{s_l_nc})}\text{ and }\mathcal{G}_{(\ref{m_s_l})}=\{(XY^{\mathbb{T}},S):(X,Y,S)\in\mathcal{G}_{(\ref{s_l_nc})}\}.
	\end{equation}
\end{proposition}
\begin{proof}
	Let $\Gamma_1:=\mathbb{R}^{m\times n}\times\mathcal{B}^S$ and $\Gamma_2:=\mathbb{R}^{m\times d}\times\mathbb{R}^{n\times d}$. By $\{(X,Y)\in\Gamma_2:\text{rank}(XY^{\mathbb{T}})\leq d\}=\mathbb{R}^{m\times d}\times\mathbb{R}^{n\times d}$ and the relation of $\text{rank}(\cdotp)$ and $\text{nnzc}(\cdotp)$, we have
	\begin{align*}
		&\min_{(Z,S)\in\Gamma_1:\,\text{rank}(Z)\leq d}\{f(Z,S)+2\lambda\,\text{rank}(Z)+\beta\,\|S\|_0\}\\
		=&\min_{(Z,S)\in\Gamma_1:\,{\text{rank}(Z)\leq d}}\big\{f(Z,S)+\lambda\min_{(X,Y)\in\Gamma_2:XY^{\mathbb{T}}=Z}\{\text{nnzc}(X)+\text{nnzc}(Y)\}+\beta\,\|S\|_0\big\}\\
		=&\min_{(X,Y,Z,S)\in\Gamma_2\times\Gamma_1:\,\text{rank}(Z)\leq d,\,XY^{\mathbb{T}}=Z}\{f(Z,S)+\lambda(\text{nnzc}(X)+\text{nnzc}(Y))+\beta\,\|S\|_0\}\\
		=&\min_{(X,Y,S)\in\Gamma_2\times\mathcal{B}^S}\{f(XY^{\mathbb{T}},S)+\lambda\big(\text{nnzc}(X)+\text{nnzc}(Y)\big)+\beta\,\|S\|_0\}.
	\end{align*}
	This implies that (\ref{lf_gl_transf}) holds. 
\end{proof}

Next, we establish the relations on global minimizers among (\ref{m_s_l}), (\ref{m_s_l_b}) and (\ref{s_l}).
\begin{proposition}\label{gl_re_c}
	Assume that $\tau\geq\|\mathcal{B}^Z\|_{\max}^{1/2}$ in problem (\ref{s_l}). If $(\mathcal{B}^Z\times \mathcal{B}^S)\cap\,\mathcal{G}_{(\ref{m_s_l})}\neq\emptyset$, then 
	\begin{equation}\label{rels-mods}
		V_{(\ref{m_s_l_b})}=V_{(\ref{s_l})}=V_{(\ref{m_s_l})}\text{ and }\mathcal{G}_{(\ref{m_s_l_b})}\subseteq\{(XY^{\mathbb{T}},S):(X,Y,S)\in\mathcal{G}_{(\ref{s_l})}\}\subseteq\mathcal{G}_{(\ref{m_s_l})}.
	\end{equation}
\end{proposition}
\begin{proof}
	By \cite[Proposition 2]{Li2024} and $\tau\geq\|\mathcal{B}^Z\|_{\max}^{1/2}$, we have that for any $Z\in\mathcal{B}^Z$, there exists $(X,Y)\in\mathcal{B}^X\times \mathcal{B}^Y$ such that $XY^{\mathbb{T}}=Z$ and ${\text{nnzc}}(X)={\text{nnzc}}(Y)={\text{rank}}(Z)$, which implies $V_{(\ref{s_l})}\leq V_{(\ref{m_s_l_b})}$. Combining with the fact that $V_{(\ref{s_l_nc})}\leq V_{(\ref{s_l})}$, we get $V_{(\ref{s_l_nc})}\leq V_{(\ref{s_l})}\leq V_{(\ref{m_s_l_b})}$. It follows from Proposition \ref{gls_mf} that $V_{(\ref{m_s_l})}=V_{(\ref{s_l_nc})}$ and $\mathcal{G}_{(\ref{m_s_l})}=\{(XY^{\mathbb{T}},S):(X,Y,S)\in\mathcal{G}_{(\ref{s_l_nc})}\}$. Due to $\mathcal{G}_{(\ref{m_s_l})}\cap (\mathcal{B}^Z\times \mathcal{B}^S)\neq\emptyset$, we obtain $V_{(\ref{m_s_l})}=V_{(\ref{m_s_l_b})}$. Thus, we have $V_{(\ref{m_s_l})}=V_{(\ref{s_l_nc})}=V_{(\ref{s_l})}=V_{(\ref{m_s_l_b})}$. It follows from \cite[Proposition 2]{Li2024} and the relations between the constraint sets that $\mathcal{G}_{(\ref{m_s_l_b})}\subseteq\{(XY^{\mathbb{T}},S):(X,Y,S)\in\mathcal{G}_{(\ref{s_l})}\}\subseteq\{(XY^{\mathbb{T}},S):(X,Y,S)\in\mathcal{G}_{(\ref{s_l_nc})}\}=\mathcal{G}_{(\ref{m_s_l})}$.
\end{proof}

The assumptions in Proposition \ref{gl_re_c} naturally hold when the feasible region of (\ref{m_s_l_b}) is large enough. When problems (\ref{m_s_l}) and (\ref{m_s_l_b}) own the same global solution set, all global minimizers of problems (\ref{m_s_l}) and (\ref{s_l}) can be transformed into each other. Based on Proposition \ref{gls_mf} and Proposition \ref{gl_re_c}, some relations among problems (\ref{s_l_nc}), (\ref{m_s_l}), (\ref{s_l}) and (\ref{m_s_l_b}) are as shown in Figure\,\ref{model-rela}.
\begin{figure}
	\centering
	\includegraphics[width=5.7in]{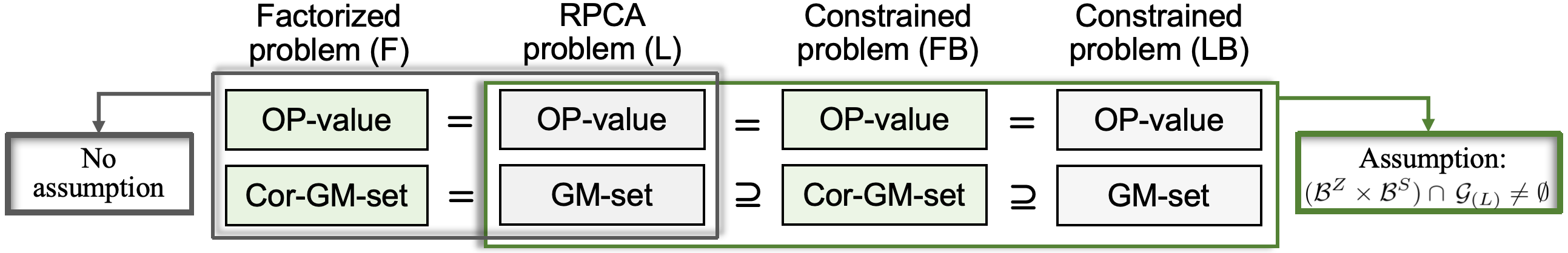}
	\caption{Relations among problems (\ref{s_l_nc}), (\ref{m_s_l}), (\ref{s_l}) and (\ref{m_s_l_b}), where `OP-value' means `optimal value', and `(Cor-)GM-set' means `the set of (corresponding matrix products of) global minimizers'.}\label{model-rela}
\end{figure}

Next, we show some properties for the local minimizers of (\ref{m_s_l_b}) and (\ref{s_l}).
\begin{proposition}\label{orig_rank_ps}
	The following properties hold for problem (\ref{m_s_l_b}).
	\begin{itemize}
		\item[(i)] $(\bar{Z},\bar{S})\in\mathcal{L}_{(\ref{m_s_l_b})}$ iff it is a local minimizer of $f(Z,S)$ on $\mathcal{B}^Z_{\lambda,\bar{Z}}\times{\underline{\mathcal{B}}}^S_{\mathbb{J}_{\bar{S}}}$, where $\mathcal{B}^Z_{\lambda,\bar{Z}}:=\{Z\in\mathcal{B}^Z:\emph{\text{rank}}(Z)\leq d,\,\lambda\,\emph{\text{rank}}(Z)= \lambda\,\emph{\text{rank}}(\bar{Z})\}$.
		\item[(ii)] If $(\bar{Z},\bar{S})\in\mathcal{L}_{(\ref{m_s_l_b})}$, then $\emph{\text{rank}}(\bar{Z})\leq d$ and
		\begin{equation}\label{lr_loc-incl}
			\exists\,\,(H_{Z},H_{S})\in\partial_{(Z,S)} f(\bar{Z},\bar{S}),\,\,\,
			\emph{s.t.}-H_{Z}\in N_{\mathcal{B}^Z_{\lambda,\bar{Z}}}(\bar{Z}),\,-H_{S}\in N_{\underline{\mathcal{B}}^S_{\mathbb{J}_{\bar{S}}}}(\bar{S}).
		\end{equation}
	\end{itemize}
\end{proposition}
\begin{proof}
	(i) The property obviously holds for problem (\ref{m_s_l_b}) with $\lambda=\beta=0$, so we analyze the case that $\lambda\neq0$ or $\beta\neq0$. Let $(\bar{Z},\bar{S})\in\mathcal{B}^Z\times\mathcal{B}^S$ be any given matrix couple. By Weyl's inequality and the property of $\ell_0$-norm, there exists an $\epsilon>0$ such that $\text{rank}(Z)\geq\text{rank}(\bar{Z})$ and $\mathbb{J}_S\supseteq\mathbb{J}_{\bar{S}}$ for any $(Z,S)\in B_{\epsilon}(\bar{Z},\bar{S})$. 
	%Define $\Omega_{\bar{Z}}:=\{(Z,S)\in\mathcal{B}^Z\times{\underline{\mathcal{B}}}^S_{\mathbb{J}_{\bar{S}}}:\lambda\,\text{rank}(Z)=\lambda\,\text{rank}(\bar{Z})\}$. 
	
	Let $(\bar{Z},\bar{S})$ be a local minimizer of $f$ on $\mathcal{B}^Z_{\lambda,\bar{Z}}\times{\underline{\mathcal{B}}}^S_{\mathbb{J}_{\bar{S}}}$. On the one hand, there exists an $\hat\epsilon>0$ such that $F_0(\bar{Z},\bar{S})\leq F_0(Z,S),\forall (Z,S)\in\{(Z,S)\in B_{\hat\epsilon}(\bar{Z},\bar{S})\cap(\mathcal{B}^Z\times\mathcal{B}^S):{\text{rank}}(Z)\leq d,\,\lambda\,{\text{rank}}(Z)=\lambda\,{\text{rank}}(\bar{Z}),\,\mathbb{J}_S=\mathbb{J}_{\bar{S}}\}$. On the other hand, it follows from the continuity of $f$ that there exists an $\bar\epsilon\in(0,\min\{\hat\epsilon,\epsilon\}]$ such that $f(\bar{Z},\bar{S})\leq f(Z,S)+\min\{2I_{+}(\lambda),I_{+}(\beta)\},\forall (Z,S)\in B_{\bar\epsilon}(\bar{Z},\bar{S})$. This means that $F_0(\bar{Z},\bar{S})\leq f(Z,S)+\min\{2I_{+}(\lambda),I_{+}(\beta)\}+2\lambda\,\text{rank}(\bar{Z})+\beta\,\|\bar{S}\|_0\leq F_0(Z,S),\forall (Z,S)\in\{(Z,S)\in B_{\bar\epsilon}(\bar{Z},\bar{S})\cap(\mathcal{B}^Z\times\mathcal{B}^S):{\text{rank}}(Z)\leq d,\,\lambda\,{\text{rank}}(Z)\neq\lambda\,{\text{rank}}(\bar{Z})\,\,\text{or}\,\,\mathbb{J}_S\neq\mathbb{J}_{\bar{S}}\}$. Thus, $(\bar{Z},\bar{S})\in\mathcal{L}_{(\ref{m_s_l_b})}$.
	
	Conversely, consider $(\bar{Z},\bar{S})\in\mathcal{L}_{(\ref{m_s_l_b})}$. Then there exists $\varepsilon\in(0,\epsilon]$ such that
	\begin{equation}\label{lcmn_ieq}
		F_0(\bar{Z},\bar{S})\leq F_0(Z,S),\forall (Z,S)\in B_\varepsilon(\bar{Z},\bar{S})\cap(\mathcal{B}^Z\times\mathcal{B}^S)\text{ with }\text{rank}(Z)\leq d,
	\end{equation}which means $f(\bar{Z},\bar{S})\leq f(Z,S),\forall (Z,S)\in B_{\varepsilon}(\bar{Z},\bar{S})\cap\mathcal{B}^Z_{\lambda,\bar{Z}}\times{\underline{\mathcal{B}}}^S_{\mathbb{J}_{\bar{S}}}$. Therefore, $(\bar{Z},\bar{S})$ is a local minimizer of $f$ on $\mathcal{B}^Z_{\lambda,\bar{Z}}\times{\underline{\mathcal{B}}}^S_{\mathbb{J}_{\bar{S}}}$.
	
	(ii) Let $(\bar{Z},\bar{S})\in\mathcal{L}_{(\ref{m_s_l_b})}$. Then, by (i), $(\bar{Z},\bar{S})$ is a local minimizer of $f(Z,S)$ on $\mathcal{B}^Z_{\lambda,\bar{Z}}\times{\underline{\mathcal{B}}}^S_{\mathbb{J}_{\bar{S}}}$. It follows from \cite[Theorem 10.1]{Rockafellar1998} and \cite[Exercise 10.10]{Rockafellar1998} that $\bm{0}\in \partial_{(Z,S)}f(\bar{Z},\bar{S})+N_{\mathcal{B}^Z_{\lambda,\bar{Z}}}(\bar{Z})\times N_{\underline{\mathcal{B}}^S_{\mathbb{J}_{\bar{S}}}}(\bar{S})$. Then, (\ref{lr_loc-incl}) holds by \cite[Theorem 6.23]{Penot2013}.
\end{proof}

\begin{proposition}\label{s_l_props}
	Problem (\ref{s_l}) owns the following properties.
	\begin{itemize}
		\item[(i)] If $(X,Y,S)\in\mathcal{G}_{(\ref{s_l})}$, then
		\begin{equation}\label{I-eq}
			\tag{C1}
			\mathbb{I}_{X}=\mathbb{I}_{Y}.
		\end{equation}
		\item[(ii)] $(\bar{X},\bar{Y},\bar{S})\in\mathcal{L}_{(\ref{s_l})}$ iff it is a local minimizer of $f(XY^{\mathbb{T}},S)$ on ${\underline{\mathcal{B}}}^X_{\mathbb{I}_{\bar{X}}}\times{\underline{\mathcal{B}}}^Y_{\mathbb{I}_{\bar{Y}}}\times{\underline{\mathcal{B}}}^S_{\mathbb{J}_{\bar{S}}}$.
		\item[(iii)] If $(\bar{X},\bar{Y},\bar{S})\in\mathcal{L}_{(\ref{s_l})}$, then
		\begin{equation}\label{loc-incl}
			\begin{split}
				\exists\,\,\,\,\,&(H_Z,H_S)\in\partial_{(Z,S)}f(\bar{X}\bar{Y}^{\mathbb{T}},\bar{S}),\\
				\emph{s.t.}\,\,&\bm{0}\in H_Z\bar{Y}+N_{{\underline{\mathcal{B}}}^X_{\mathbb{I}_{\bar{X}}}}(\bar{X}),\,\bm{0}\in H_Z^{\mathbb{T}}\bar{X}+N_{{\underline{\mathcal{B}}}^Y_{\mathbb{I}_{\bar{Y}}}}(\bar{Y}),\,\bm{0}\in H_S+N_{{\underline{\mathcal{B}}}^S_{\mathbb{J}_{\bar{S}}}}(\bar{S}).
			\end{split}
		\end{equation}
	\end{itemize}
\end{proposition}
\begin{proof}
	(i) For any $(X,Y,S)\in\mathcal{G}_{(\ref{s_l})}$, let $\mathbb{I}=\mathbb{I}_{X}\cap\mathbb{I}_{Y}$. If $\mathbb{I}_{X}\neq\mathbb{I}_{Y}$, then $F(\underline{X}_{\mathbb{I}},\underline{Y}_{\mathbb{I}},S)<F(X,Y,S)$, which makes a contradiction. Thus, $\mathbb{I}_{X}=\mathbb{I}_{Y}$.
	
	(ii) The statement obviously holds for the case $\lambda=\beta=0$, so we consider the case that $\lambda\neq0$ or $\beta\neq0$. For any given $(\bar{X},\bar{Y},\bar{S})\in\mathcal{B}:=\mathcal{B}^X\times\mathcal{B}^Y\times\mathcal{B}^S$, define $\underline{\mathcal{B}}:={\underline{\mathcal{B}}}^X_{\mathbb{I}_{\bar{X}}}\times{\underline{\mathcal{B}}}^Y_{\mathbb{I}_{\bar{Y}}}\times{\underline{\mathcal{B}}}^S_{\mathbb{J}_{\bar{S}}}$, it holds that there exists an $\varepsilon>0$ such that
	\begin{align}\label{IJ=}
		\mathbb{I}_{X}=\mathbb{I}_{\bar{X}},\mathbb{I}_{Y}=\mathbb{I}_{\bar{Y}}\text{ and }\mathbb{J}_{S}=\mathbb{J}_{\bar{S}},\,\forall (X,Y,S)\in B_{\varepsilon}(\bar{X},\bar{Y},\bar{S})\cap\underline{\mathcal{B}},
	\end{align}
	and hence
	\begin{align}\label{IJ-incl}
		\mathbb{I}_{X}\supsetneq\mathbb{I}_{\bar{X}},\mathbb{I}_{Y}\supsetneq\mathbb{I}_{\bar{Y}}\text{ or }\mathbb{J}_{S}\supsetneq\mathbb{J}_{\bar{S}},\,\forall(X,Y,S)\in B_{\varepsilon}(\bar{X},\bar{Y},\bar{S})\cap(\mathcal{B}\setminus\underline{\mathcal{B}}).
	\end{align}
	
	If $(\bar{X},\bar{Y},\bar{S})\in\mathcal{L}_{(\ref{s_l})}$, then there exists an $\varepsilon_1\in(0,\varepsilon]$ such that $F(\bar{X},\bar{Y},\bar{S})\leq F(X,Y,S)$, $\forall (X,Y,S)\in B_{\varepsilon_1}(\bar{X},\bar{Y},\bar{S})\cap\underline{\mathcal{B}}$. It follows from (\ref{IJ=}) that $f(\bar{X}\bar{Y}^{\mathbb{T}},\bar{S})\leq f(XY^{\mathbb{T}},S),\forall (X,Y,S)\in B_{\varepsilon_1}(\bar{X},\bar{Y},\bar{S})\cap\underline{\mathcal{B}}$,
	and hence 
	\begin{align}\label{loc1}
		\text{$(\bar{X},\bar{Y},\bar{S})$ is a local minimizer of $f(XY^{\mathbb{T}},S)$ on $\underline{\mathcal{B}}$.}
	\end{align}
	
	Conversely, consider that (\ref{loc1}) holds. Then there exists an $\varepsilon_2\in(0,\varepsilon]$ such that $f(\bar{X}\bar{Y}^{\mathbb{T}},\bar{S})\leq f(XY^{\mathbb{T}},S),$ $\forall (X,Y,S)\in B_{\varepsilon_2}(\bar{X},\bar{Y},\bar{S})\cap\underline{\mathcal{B}}$.
	By (\ref{IJ=}), we have
	\begin{equation}\label{ieqd4}
		F(\bar{X},\bar{Y},\bar{S})\leq F(X,Y,S),\forall (X,Y,S)\in B_{\varepsilon_2}(\bar{X},\bar{Y},\bar{S})\cap\underline{\mathcal{B}}.
	\end{equation}
	Considering the continuity of $f(XY^{\mathbb{T}},S)$, we have that there exists an $\varepsilon_3\in(0,\varepsilon_2]$ such that $f(\bar{X}\bar{Y}^{\mathbb{T}},\bar{S})\leq f(XY^{\mathbb{T}},S)+\min\{I_{+}(\lambda),I_{+}(\beta)\},\forall (X,Y,S)\in B_{\varepsilon_3}(\bar{X},\bar{Y},\bar{S})$. Further, by (\ref{IJ-incl}), for any $(X,Y,S)\in B_{\varepsilon_3}(\bar{X},\bar{Y},\bar{S})\cap(\mathcal{B}\setminus\underline{\mathcal{B}})$, we have that $F(\bar{X},\bar{Y},\bar{S})\leq \,f(XY^{\mathbb{T}},S)+\min\{I_{+}(\lambda),I_{+}(\beta)\}+\lambda\big(\text{nnzc}(\bar{X})+\text{nnzc}(\bar{Y})\big)+\beta\|\bar{S}\|_0
	\leq\,F(X,Y,S)$.
	Combining with (\ref{ieqd4}), we have $F(\bar{X},\bar{Y},\bar{S})\leq F(X,Y,S),\forall (X,Y,S)\in B_{\varepsilon_3}(\bar{X},\bar{Y},\bar{S})\cap\mathcal{B}$. As a result, $(\bar{X},\bar{Y},\bar{S})\in\mathcal{L}_{(\ref{s_l})}$.
	
	(iii) Let $(\bar{X},\bar{Y},\bar{S})\in\mathcal{L}_{(\ref{s_l})}$. By (ii), $(\bar{X},\bar{Y},\bar{S})$ is a local minimizer of $f(XY^{\mathbb{T}},S)$ on ${\underline{\mathcal{B}}}^X_{\mathbb{I}_{\bar{X}}}\times{\underline{\mathcal{B}}}^Y_{\mathbb{I}_{\bar{Y}}}\times{\underline{\mathcal{B}}}^S_{\mathbb{J}_{\bar{S}}}$. According to Fermat's rule in \cite[Theorem 10.1]{Rockafellar1998} and \cite[Exercise 10.10]{Rockafellar1998}, we get $\bm{0}\in\partial_{(X,Y,S)}f(\bar{X}\bar{Y}^{\mathbb{T}},\bar{S})+N_{{\underline{\mathcal{B}}}^X_{\mathbb{I}_{\bar{X}}}\times{\underline{\mathcal{B}}}^Y_{\mathbb{I}_{\bar{Y}}}\times{\underline{\mathcal{B}}}^S_{\mathbb{J}_{\bar{S}}}}(\bar{X},\bar{Y},\bar{S})$. Combining with \cite[Theorem 6.23]{Penot2013} and \cite[Proposition 10.5]{Rockafellar1998}, we obtain (\ref{loc-incl}).
\end{proof}

By Proposition \ref{s_l_props} (ii), it is not hard to deduce that the property (\ref{I-eq}) is a necessary condition for the global minimizers of problem (\ref{s_l}), but not for its all local minimizers in general.

\begin{definition}\label{sd_defs}
	\item[(i)] $(\bar{Z},\bar{S})\in\mathcal{B}^Z\times\mathcal{B}^S$ is called a stationary point of (\ref{m_s_l_b}) if it satisfies (\ref{lr_loc-incl}).
	\item[(ii)] $(\bar{X},\bar{Y},\bar{S})\in\mathcal{B}^X\times\mathcal{B}^Y\times\mathcal{B}^S$ is called a stationary point of (\ref{s_l}) if it satisfies (\ref{loc-incl}).\\
	The sets of all stationary points of (\ref{m_s_l_b}) and (\ref{s_l}) are denoted as $\mathcal{S}_{(\ref{m_s_l_b})}$ and $\mathcal{S}_{(\ref{s_l})}$, respectively.
\end{definition}

Note that problem (\ref{m_s_l_b}) with $\mathcal{B}^Z=\mathbb{R}^{m\times n}$, and problem (\ref{s_l}) with $\tau=+\infty$, correspond to problems (\ref{m_s_l}) and (\ref{s_l_nc}), respectively. Hence, Proposition \ref{orig_rank_ps}, Proposition \ref{s_l_props} and Definition \ref{sd_defs} are applicable to (\ref{m_s_l}) and (\ref{s_l_nc}). In what follows, we show the relations on local minimizers and stationary points between (\ref{m_s_l}) and (\ref{s_l_nc}).

\begin{theorem}\label{md_eq_lc_sd}
	Problems (\ref{m_s_l}) and (\ref{s_l_nc}) own the following relations on their local minimizers and stationary points.
	\begin{itemize}
		\item[(i)] $\mathcal{L}_{(\ref{m_s_l})}\subseteq\{(XY^{\mathbb{T}},S):(X,Y,S)\in\mathcal{L}_{(\ref{s_l_nc})}\}$ and $\mathcal{S}_{(\ref{m_s_l})}\subseteq\{(XY^{\mathbb{T}},S):(X,Y,S)\in\mathcal{S}_{(\ref{s_l_nc})}\}$.
		\item[(ii)] When $\lambda>0$, $\mathcal{S}_{(\ref{m_s_l})}=\{(XY^{\mathbb{T}},S):(X,Y,S)\in\mathcal{S}_{(\ref{s_l_nc})}\}$.
	\end{itemize}
\end{theorem}
\begin{proof}
	(i) For any given $(\bar{Z},\bar{S})\in\mathcal{L}_{(\ref{m_s_l})}$, there exists an $\epsilon>0$ such that (\ref{lcmn_ieq}) holds with $\mathcal{B}^Z=\mathbb{R}^{m\times n}$, and it is not hard to construct $(\bar{X},\bar{Y})\in\mathbb{R}^{m\times d}\times\mathbb{R}^{n\times d}$ such that $\text{nnzc}(\bar{X})=\text{nnzc}(\bar{Y})=\text{rank}(\bar{Z})$ and $\bar{X}\bar{Y}^{\mathbb{T}}=\bar{Z}$. By the continuity of $\|XY^{\mathbb{T}}\|_F$, there exists an $\epsilon_r\in(0,\epsilon)$ such that $XY^{\mathbb{T}}\in B_{\epsilon}(\bar{Z}),\forall (X,Y)\in B_{\epsilon_r}(\bar{X},\bar{Y})$. Combining with (\ref{lcmn_ieq}) and $\text{rank}(XY^{\mathbb{T}})\leq d$, we have $F(\bar{X},\bar{Y},\bar{S})=F_0(\bar{Z},\bar{S})\leq F_0(XY^{\mathbb{T}},S)\leq F(X,Y,S),\forall (X,Y,S)\in B_{\epsilon_r}(\bar{X},\bar{Y})\times (B_{\epsilon-\epsilon_r}(\bar{S})\cap\mathcal{B}^S)$. Thus, $(\bar{X},\bar{Y},\bar{S})\in\mathcal{L}_{(\ref{s_l_nc})}$.
	
	For any $\bar{Z}\in\mathbb{R}^{m\times n}$ with $\text{rank}(\bar{Z})\leq d$, let $U_{\bar{Z}}^{\bot}\in\mathbb{R}^{m\times(m-d_{\bar{Z}})}$ and $V_{\bar{Z}}^{\bot}\in\mathbb{R}^{n\times(n-d_{\bar{Z}})}$ satisfy that $[U_{\bar{Z}},U_{\bar{Z}}^{\bot}]$ and $[V_{\bar{Z}},V_{\bar{Z}}^{\bot}]$ are orthogonal matrices. Set $N_{\bar{Z}}:=\big\{U_{\bar{Z}}^{\bot}M(V_{\bar{Z}}^{\bot})^{\mathbb{T}}:M\in\mathbb{R}^{(m-d_{\bar{Z}})\times(n-d_{\bar{Z}})}\big\}$. By the definition of $\mathcal{B}^Z_{\lambda,\bar{Z}}$ in Proposition \ref{orig_rank_ps}, we have that $\mathcal{B}^Z_{\lambda,\bar{Z}}=\{Z\in\mathcal{B}^Z:\text{rank}(Z)= \text{rank}(\bar{Z})\}$ if $\lambda>0$, and $\mathcal{B}^Z_{\lambda,\bar{Z}}=\{Z\in\mathcal{B}^Z:\text{rank}(Z)\leq d\}$ if $\lambda=0$. It follows from \cite[Proposition 3.6]{luke2013prox} that for any $\bar{Z}\in\mathbb{R}^{m\times n}$ with $\text{rank}(\bar{Z})\leq d$,
	\begin{equation}\label{normalc}
		\text{$N_{\mathcal{B}^Z_{\lambda,\bar{Z}}}(\bar{Z})=N_{\bar{Z}}$ if $\lambda>0$, and $N_{\mathcal{B}^Z_{\lambda,\bar{Z}}}(\bar{Z})\subseteq N_{\bar{Z}}$ if $\lambda=0$.}
	\end{equation}
	
	Next, consider $(\bar{Z},\bar{S})\in\mathcal{S}_{(\ref{m_s_l})}$. Define $\bar{X}:=[U_{\bar{Z}}\Sigma_{\bar{Z}}^{1/2},\bm{0}^{m\times(d-d_{\bar{Z}})}]$, $\bar{Y}:=[V_{\bar{Z}}\Sigma_{\bar{Z}}^{1/2},$ $\bm{0}^{n\times(d-d_{\bar{Z}})}]$. Combining (\ref{normalc}) and (\ref{lr_loc-incl}), there exist $(H_Z,H_S)\in\partial_{(Z,S)}f(\bar{Z},\bar{S})$ and $\bar{M}\in\mathbb{R}^{(m-d_{\bar{Z}})\times(n-d_{\bar{Z}})}$ such that $\bm{0}=H_Z+U_{\bar{Z}}^{\bot}\bar{M}(V_{\bar{Z}}^{\bot})^{\mathbb{T}}$ and $\bm{0}\in H_S+N_{\underline{\mathcal{B}}^S_{\mathbb{J}_{\bar{S}}}}(\bar{S})$. Further, by multiplying $\bar{X}^{\mathbb{T}}$ on the first equation from the left and $\bar{Y}$ on it from the right, respectively, we get $H_Z^{\mathbb{T}}\bar{X}=\bm{0}$ and $H_Z\bar{Y}=\bm{0}$. Combining with $\mathcal{B}^X=\mathbb{R}^{m\times d}$, $\mathcal{B}^Y=\mathbb{R}^{n\times d}$ and $\bm{0}\in H_S+N_{\underline{\mathcal{B}}^S_{\mathbb{J}_{\bar{S}}}}(\bar{S})$, we get (\ref{loc-incl}). Thus, $(\bar{X},\bar{Y},\bar{S})\in\mathcal{S}_{(\ref{s_l_nc})}$.
	
	(ii) Consider $(\bar{X},\bar{Y},\bar{S})\in\mathcal{S}_{(\ref{s_l_nc})}$ and $\lambda>0$. Based on $\mathcal{B}^X=\mathbb{R}^{m\times d}$, $\mathcal{B}^Y=\mathbb{R}^{n\times d}$ and (\ref{loc-incl}), there exists an $(H_Z,H_S)\in\partial_{(Z,S)}f(\bar{X}\bar{Y}^{\mathbb{T}},\bar{S})$ such that $\bm{0}=H_Z\bar{Y}_{\mathbb{I}_{\bar{X}}\cap\mathbb{I}_{\bar{Y}}}$, $\bm{0}=H_Z^{\mathbb{T}}\bar{X}_{\mathbb{I}_{\bar{X}}\cap\mathbb{I}_{\bar{Y}}}$ and $-H_S\in N_{{\underline{\mathcal{B}}}^S_{\mathbb{J}_{\bar{S}}}}(\bar{S})$. Let $\bar{Z}=\bar{X}\bar{Y}^{\mathbb{T}}$ and $\bar{d}=\text{rank}(\bar{Z})$. Then, $\bar{Z}=\bar{X}_{\mathbb{I}_{\bar{X}}\cap\mathbb{I}_{\bar{Y}}}\bar{Y}_{\mathbb{I}_{\bar{X}}\cap\mathbb{I}_{\bar{Y}}}^{\mathbb{T}}$ and hence $\text{rank}(\bar{X}_{\mathbb{I}_{\bar{X}}\cap\mathbb{I}_{\bar{Y}}})\geq\bar{d}$, $\text{rank}(\bar{Y}_{\mathbb{I}_{\bar{X}}\cap\mathbb{I}_{\bar{Y}}})\geq\bar{d}$, $\text{rank}(H_Z)\leq \min\{m,n\}-\bar{d}$, $\bar{Z}H_Z^{\mathbb{T}}=\bar{X}_{\mathbb{I}_{\bar{X}}\cap\mathbb{I}_{\bar{Y}}}(H_Z\bar{Y}_{\mathbb{I}_{\bar{X}}\cap\mathbb{I}_{\bar{Y}}})^{\mathbb{T}}=\bm{0}$ and $H_Z^{\mathbb{T}}\bar{Z}=(H_Z^{\mathbb{T}}\bar{X}_{\mathbb{I}_{\bar{X}}\cap\mathbb{I}_{\bar{Y}}})$ $\bar{Y}^{\mathbb{T}}_{\mathbb{I}_{\bar{X}}\cap\mathbb{I}_{\bar{Y}}}=\bm{0}$. Denote the compact SVDs of $\bar{Z}$ and $H_Z$ as $\bar{Z}=U_{\bar{Z}}\Sigma_{\bar{Z}}V_{\bar{Z}}^{\mathbb{T}}$ and $H_Z=U_{H_Z}\Sigma_{H_Z}V_{H_Z}^{\mathbb{T}}$, respectively. Combining with $\bar{Z}H_Z^{\mathbb{T}}=\bm{0}$, $H_Z^{\mathbb{T}}\bar{Z}=\bm{0}$, the orthogonality of $U_{\bar{Z}},V_{\bar{Z}},U_{H_Z}$, $V_{H_Z}$, and the invertibility of $\Sigma_{\bar{Z}},\Sigma_{H_Z}$, we deduce that $V_{\bar{Z}}^{\mathbb{T}}V_{H_Z}=\bm{0}$ and $U_{H_Z}^{\mathbb{T}}U_{\bar{Z}}=\bm{0}$. Since $\text{rank}(H_Z)\leq \min\{m,n\}-\bar{d}$, the column numbers of $U_{H_Z}$ and $V_{H_Z}$ are not more than $\min\{m,n\}-\bar{d}$, which implies that $-H_Z\in N_{\mathcal{B}^Z_{\lambda,\bar{Z}}}(\bar{Z})$ by (\ref{normalc}). Combining with $-H_S\in N_{{\underline{\mathcal{B}}}^S_{\mathbb{J}_{\bar{S}}}}(\bar{S})$, we obtain that $(\bar{Z},\bar{S})\in\mathcal{S}_{(\ref{m_s_l})}$. This together with the second statement of (i) yields that $\mathcal{S}_{(\ref{m_s_l})}=\{(XY^{\mathbb{T}},S):(X,Y,S)\in\mathcal{S}_{(\ref{s_l_nc})}\}$.
\end{proof}

\begin{remark}
	Proposition \ref{gls_mf} and Theorem \ref{md_eq_lc_sd} show the relations between problems (\ref{m_s_l}) and (\ref{s_l_nc}) in the sense of global minimizers, local minimizers, and stationary points. This implies that the double equivalence on global minimizers and stationary points is applicable to problem (\ref{m_s_l}) with $\lambda>0$ and $\beta\geq0$, which includes the rank regularized problem $\min_{Z:\,\emph{\text{rank}}(Z)\leq d}\{f(Z)+\lambda\,\emph{\text{rank}}(Z)\}$ in \cite{Li2024} as a special case.
\end{remark}

By Theorem \ref{md_eq_lc_sd} (ii), it is not hard to check that when $\lambda>0$, if $(\bar{X},\bar{Y},\bar{S})\in\mathcal{S}_{(\ref{s_l})}$ satisfies $(\bar{X}_{\mathbb{I}_{\bar{X}}},\bar{Y}_{\mathbb{I}_{\bar{Y}}})\in\emph{\emph{int}}(\mathcal{B}^X_{\mathbb{I}_{\bar{X}}}\times\mathcal{B}^Y_{\mathbb{I}_{\bar{Y}}})$, then $(\bar{X}\bar{Y}^{\mathbb{T}},\bar{S})\in\mathcal{S}_{(\ref{m_s_l})}$. Based on Proposition \ref{gls_mf} and Theorem \ref{md_eq_lc_sd}, some relations between problem (\ref{m_s_l}) and its factorization problem (\ref{s_l_nc}) are as shown in Figure\,\ref{model-rela2}.
\begin{figure}
	\centering
	\includegraphics[width=4.3in]{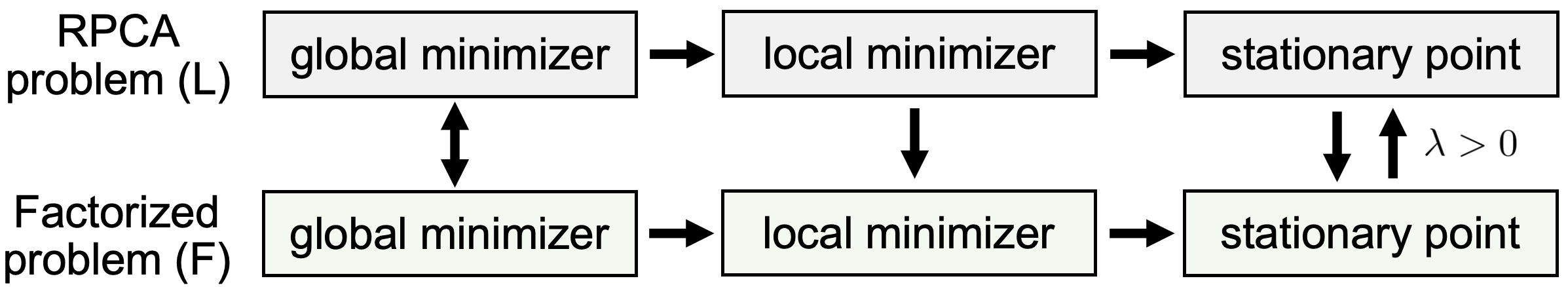}
	\caption{Relations between problems (\ref{m_s_l}) and (\ref{s_l_nc}).}\label{model-rela2}
\end{figure}

\section{Equivalent continuous relaxations of problem (\ref{s_l})}\label{section 3}
For some common matrix recovery problems, such as image restoration, score sheet completion and collaborative filtering, the feasible regions of the involved optimization problems are usually bounded from their practical background. Combining with the fact that the low-rank structure provides computational benefits for the matrix factorization model, we focus on the constrained problem (\ref{s_l}) in this section. We construct an equivalent continuous nonconvex relaxation framework for (\ref{s_l}) and mention the boundedness requirement of its feasible region in theoretical analysis when appropriate. Under this framework, all relaxation problems of (\ref{s_l}) own the same global minimizers and optimal value. Furthermore, we analyze the optimality conditions of these relaxation problems and establish the equivalence between (\ref{s_l}) and these relaxations in the sense of strong stationary points defined in Definition \ref{ssp_defs}.

Let $\theta:\mathbb{R}_+\rightarrow\mathbb{R}_+$ be any given function satisfying
\begin{equation}\label{theta_cond}
	\theta(0)=0,\,\,\theta(t)\in[0,1],\forall t\in(0,1),\,\,\theta(t)=1,\forall t\in[1,+\infty).
\end{equation}
Its function curve diagram is shown in Figure\,\ref{theta_fig}.
\begin{figure}[!t]
	\centering
	\includegraphics[width=1.7in]{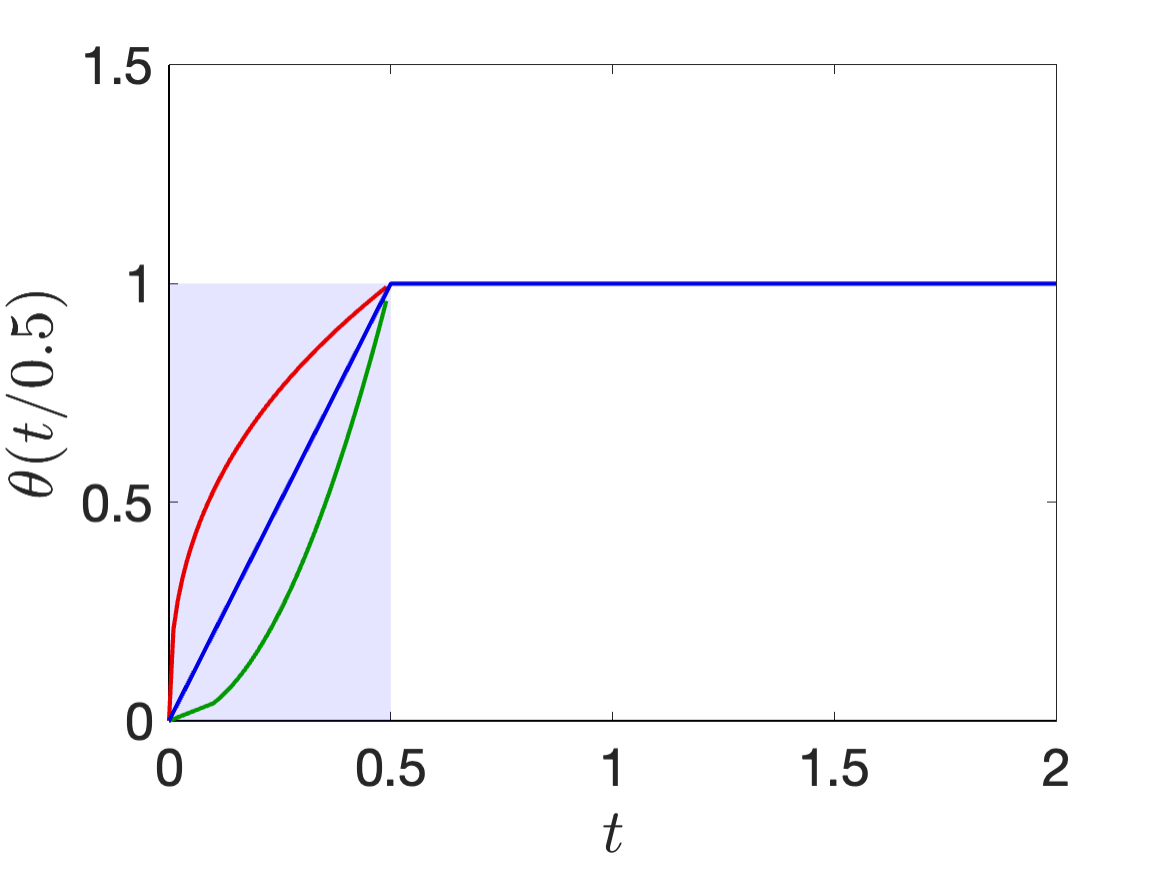}
	\hfil
	\includegraphics[width=1.7in]{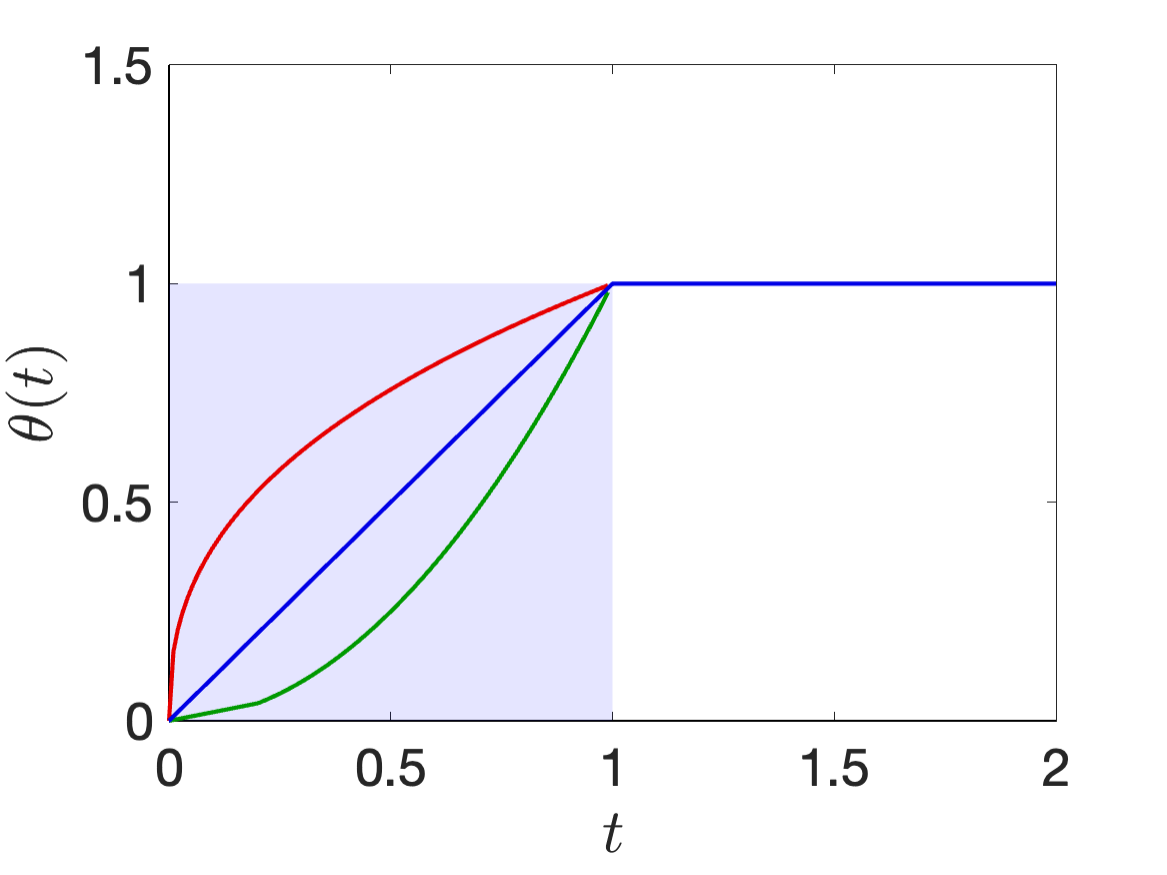}
	\hfil
	\includegraphics[width=1.7in]{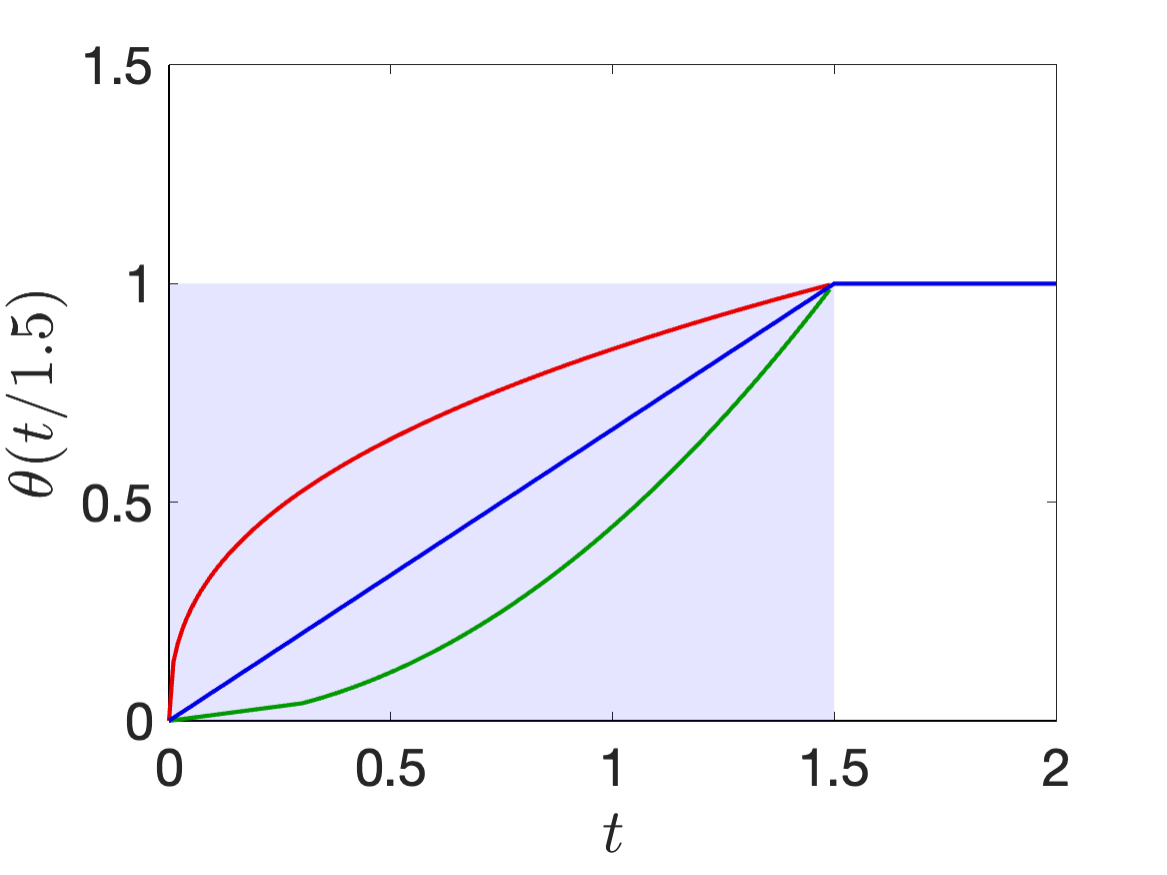}
	\caption{Function curve diagrams for $\theta(\cdot)$ satisfying (\ref{theta_cond}) and its scaled forms $\theta(\cdot/0.5)$ and $\theta(\cdot/1.5)$, where the red, blue and green curves represent three applicable alternatives to $\theta$ when $t\in(0,1)$.}
	\label{theta_fig}
\end{figure}

Relying on the given function $\theta$ and parameter $\rho>0$, we define $\Theta^\rho:\mathbb{R}^{r\times d}\rightarrow\mathbb{R}_+$ and $\tilde\Theta^\rho:\mathbb{R}^{m\times n}\rightarrow\mathbb{R}_+$ as
\begin{equation}\label{Theta-def}
	\Theta^\rho(V):=\sum_{i=1}^d\theta(\|V_i\|_2/\rho)\mbox{ and }\tilde\Theta^\rho(S):=\sum_{i=1}^m\sum_{j=1}^n\theta(|S_{ij}|/\rho).
\end{equation}
Obviously, $\lim_{\rho\downarrow0}\Theta^\rho(V)=\text{nnzc}(V),\,\forall V\in\mathbb{R}^{r\times d}$ and $\lim_{\rho\downarrow0}\tilde\Theta^\rho(S)=\|S\|_0,\,\forall S\in\mathbb{R}^{m\times n}$.

For the given function $\theta$ and parameters $r,s>0$, define the function $F_{\theta}^{r,s}:\mathbb{R}^{m\times d}\times\mathbb{R}^{n\times d}\times\mathbb{R}^{m\times n}\rightarrow\mathbb{R}_+$ as $$F_{\theta}^{r,s}(X,Y,S):=f(XY^{\mathbb{T}},S)+\lambda\big(\Theta^{r}(X)+\Theta^{r}(Y)\big)+\beta\tilde\Theta^{s}(S).$$
Correspondingly, we consider the following  relaxation of problem (\ref{s_l}),
\begin{equation}\label{s_l_r}
	\tag{FR}
	\min_{(X,Y,S)\in \mathcal{B}}F_{\theta}^{r,s}(X,Y,S).
\end{equation}

To highlight various choices of parameters $r$, $s$ and function $\theta$ in problem (\ref{s_l_r}), we represent its optimal value and global solution set as $V^{\theta:r,s}_{(\ref{s_l_r})}$ and $\mathcal{G}^{\theta:r,s}_{(\ref{s_l_r})}$, respectively. Next, we show the equivalence on the global minimizers between problems (\ref{s_l}) and (\ref{s_l_r}) under the following assumption. 
\begin{assumption}\label{glob_para}
	For any $(X,Y,S)\in\mathcal{G}^{\theta:r,s}_{(\ref{s_l_r})}$, it holds that
	\begin{equation}\label{glb_prp}
		\tag{C2}
		\|X_i\|_2,\|Y_i\|_2\notin(0,r),\,\forall\, i\in[d]\text{ if }\lambda>0,\text{ and }|S_{ij}|\notin(0,s),\,\forall\, (i,j)\in[m]\times[n]\text{ if }\beta>0.
	\end{equation}
\end{assumption}
\begin{proposition}\label{glob_eqv}
	If Assumption \ref{glob_para} holds, one has that 
	\begin{equation*}
		\text{$\mathcal{G}_{(\ref{s_l})}=\mathcal{G}^{\theta:r,s}_{(\ref{s_l_r})}$ and $V_{(\ref{s_l})}=V^{\theta:r,s}_{(\ref{s_l_r})}$.}
	\end{equation*}
\end{proposition}
\begin{proof}
	Use the symbol $W$ to abbreviate $(X,Y,S)$ and define $\mathcal{D}:=\{W\in \mathcal{B}:W\text{ satisfies }(\ref{glb_prp})\}$. Recall that $\vartheta$ is defined as (\ref{c_def}). Since $\theta(t)\leq\vartheta(t),\forall t\in\mathbb{R}_+$, we have $F_{\theta}^{r,s}(W)\leq F(W),\forall W\in\mathcal{B}$. Combining with $\mathcal{D}\subseteq\mathcal{B}$, we get 
	\begin{equation}\label{mobj_ieq}
		V^{\theta:r,s}_{(\ref{s_l_r})}\leq V_{(\ref{s_l})}\leq\min_{W\in\mathcal{D}}F(W).
	\end{equation}
	Due to Assumption \ref{glob_para}, we obtain that $\mathcal{G}^{\theta:r,s}_{(\ref{s_l_r})}\subseteq\mathcal{D}$ and $F_{\theta}^{r,s}(W)=F(W),\forall W\in\mathcal{D}$. Combining with $\mathcal{D}\subseteq\mathcal{B}$ and (\ref{mobj_ieq}), we have
	\begin{align}\label{mobj_ieq2}
		V^{\theta:r,s}_{(\ref{s_l_r})}=\min_{W\in\mathcal{D}}F_{\theta}^{r,s}(W)=\min_{W\in\mathcal{D}}F(W)=V_{(\ref{s_l})}.
	\end{align}
	Then we deduce that $\mathcal{G}^{\theta:r,s}_{(\ref{s_l_r})}=\text{argmin}_{W\in\mathcal{D}}F_{\theta}^{r,s}(W)=\text{argmin}_{W\in\mathcal{D}}F(W)\subseteq\mathcal{G}_{(\ref{s_l})}$. Moreover, by (\ref{mobj_ieq2}) and $F_{\theta}^{r,s}(W)\leq F(W),\forall W\in\mathcal{B}$, we have $\mathcal{G}_{(\ref{s_l})}\subseteq\mathcal{G}^{\theta:r,s}_{(\ref{s_l_r})}$. In conclusion, $\mathcal{G}_{(\ref{s_l})}=\mathcal{G}^{\theta:r,s}_{(\ref{s_l_r})}$.
\end{proof}

It is not hard to verify that if $\theta:\mathbb{R}_+\rightarrow\mathbb{R}_+$ satisfies Assumption \ref{glob_para}, then for any $\bar\theta:\mathbb{R}_+\rightarrow\mathbb{R}_+$ satisfying
\begin{equation*}
	\theta(t)\leq\bar\theta(t)\leq\vartheta(t),\,\forall t\in\mathbb{R}_+,\,\text{with $\vartheta$ defined as (\ref{c_def}),}
\end{equation*}
one has that $\mathcal{G}_{(\ref{s_l})}=\mathcal{G}^{\bar\theta:r,s}_{(\ref{s_l_r})}$ and $V_{(\ref{s_l})}=V^{\bar\theta:r,s}_{(\ref{s_l_r})}$. Thus, based on an appropriate $\theta$ as in Assumption \ref{glob_para}, we can get an infinite variety of equivalent relaxation problems for problem (\ref{s_l}).

Define functions $\theta_1,\theta_2:\mathbb{R}_+\rightarrow\mathbb{R}_+$ as follows:
\begin{equation}\label{theta1_def}
	\theta_1(t)=t,\forall t\in\{0,1\},\,\,\theta_1(t)\in(0,1),\forall\, t\in(0,1),\,\,\theta_1(t)>1,\forall t\in(1,+\infty),
\end{equation}
and $\theta_2(t)=1,\forall\, t\in\mathbb{R}_+$.
Note that relaxation function $\theta$ satisfying (\ref{theta_cond}) can be equivalently expressed in the following form
\begin{equation}\label{theta_ec}
	\theta(t)=\min\{\theta_1(t),\theta_2(t)\}.
\end{equation}
It follows that $\theta(t)=\theta_1(t),\forall t\in[0,1]\text{ and }\theta(t)=\theta_2(t),\forall t\in[1,+\infty)$.
For the convenience of subsequent analysis, we employ the expression (\ref{theta_ec}) for function $\theta$.

Next, we present a verifiable assumption to establish the equivalences on global minimizers and other optimality conditions between problems (\ref{s_l}) and (\ref{s_l_r}).
\begin{assumption}\label{assumpars}
	\item[(i)] For any $k\in[d]$ and $(i,j)\in[m]\times[n]$, $f(XY^{\mathbb{T}},S)$ is Lipschitz continuous on $\mathcal{B}$ with respect to $X_k$ and $Y_k$, with positive Lipschitz constants $L_X$ and $L_Y$, respectively, if $\lambda>0$, and with respect to $S_{ij}$ with positive Lipschitz constant $L_S$ if $\beta>0$.
	\item[(ii)] $\theta_1$ is a locally Lipschitz continuous function on $\mathbb{R}_+$ satisfying that
	\begin{equation}\label{eta_def}
		\eta:=\min\{\eta_1\in \partial^c\theta_{1}(t):t\in(0,1)\}>0,
	\end{equation}
	and positive parameters $r$ and $s$ in problem (\ref{s_l_r}) are chosen by
	\begin{equation}\label{subd_relas}
		r<\min\big\{\tau,\lambda\eta/\max\{L_X,L_Y\}\big\}\,\,\text{if}\,\,\lambda>0,\,\,\text{and}\,\,\,s<\beta\eta/L_S\,\,\text{if}\,\,\beta>0.
	\end{equation}
\end{assumption}

According to the inclusion relation between limiting subdifferential and Clarke subdifferential in \cite[Proposition 4.3.2]{cui2021modern}, we have $\min\{\eta_1\in\partial\theta_{1}(t):t\in(0,1)\}\geq\eta$.

Note that if $\mathcal{B}$ is bounded, then the moduli $L_X$, $L_Y$ and $L_S$ in Assumption \ref{assumpars} (i) exist. Function $\theta_1$ satisfying Assumption \ref{assumpars} (ii) can be taken as common convex or nonconvex relaxation functions of $\vartheta$, such as $\ell_p$-norm $(0<p\leq 1)$, SCAD and MCP.

\begin{theorem}\label{assum-rela}
	Suppose that Assumption \ref{assumpars} holds. Then, Assumption \ref{glob_para} holds for problem (\ref{s_l_r}), and hence $\mathcal{G}_{(\ref{s_l})}=\mathcal{G}_{(\ref{s_l_r})}$ and $V_{(\ref{s_l})}=V_{(\ref{s_l_r})}$.
	Moreover, any global minimizer $(X,Y,S)$ of problems (\ref{s_l}) and (\ref{s_l_r}) owns the properties (\ref{I-eq}) and (\ref{glb_prp}).
\end{theorem}
\begin{proof}
	First, we prove that (\ref{glb_prp}) holds for any $(X,Y,S)\in\mathcal{G}^{\theta:r,s}_{(\ref{s_l_r})}$. Suppose that it does not hold. Then there exists a $(X,Y,S)\in\mathcal{G}^{\theta:r,s}_{(\ref{s_l_r})}$ that does not satisfy (\ref{glb_prp}). So, there exists $i_X\in[d]$, $i_Y\in[d]$ or $(i_S,j_S)\in[m,n]$ such that $\|X_{i_X}\|_2\in(0,r)$, $\|Y_{i_Y}\|_2\in(0,r)$ or $|S_{i_Sj_S}|\in(0,s)$.
	
	Consider the first case $\|X_{i_X}\|_2\in(0,r)$ when $\lambda>0$. Let $(\tilde{X},Y,S)\in\mathcal{B}$ with $\tilde{X}_i=X_i,\forall i\neq i_X$ and $\tilde{X}_{i_X}=\bm{0}$. By Assumption \ref{assumpars} and \cite[Theorem 2.3.7]{Clarke1983}, we get 
	\begin{equation*}
		\begin{split}
			&f(\tilde{X}Y^{\mathbb{T}},S)-f(XY^{\mathbb{T}},S)
			\leq\,L_X\|X-\tilde{X}\|_F<{\lambda\eta}\|X-\tilde{X}\|_F/{r}\\
			\leq\,&\,\lambda\big(\Theta^{r}(X)+\Theta^{r}(Y)\big)+\beta\tilde\Theta^{s}(S)-\lambda\big(\Theta^{r}(\tilde{X})+\Theta^{r}(Y)\big)-\beta\tilde\Theta^{s}(S).
		\end{split}
	\end{equation*}
	This means that $F_{\theta}^{r,s}(\tilde{X},Y,S)<F_{\theta}^{r,s}(X,Y,S)$, which contradicts $(X,Y,S)\in\mathcal{G}^{\theta:r,s}_{(\ref{s_l_r})}$. Thus, $\|X_i\|_2\notin(0,r),\forall i\in[d]$ for any $(X,Y,S)\in\mathcal{G}^{\theta:r,s}_{(\ref{s_l_r})}$.
	
	Similarly, we deduce that $\|Y_i\|_2\notin(0,r),\forall i\in[d]$ when $\lambda>0$ and $|S_{ij}|\notin(0,s),\forall(i,j)\in[m,n]$ when $\beta>0$ for any $(X,Y,S)\in\mathcal{G}^{\theta:r,s}_{(\ref{s_l_r})}$. As a result, (\ref{glb_prp}) holds for any $(X,Y,S)\in\mathcal{G}^{\theta:r,s}_{(\ref{s_l_r})}$. It follows from Proposition \ref{glob_eqv} that $\mathcal{G}_{(\ref{s_l})}=\mathcal{G}_{(\ref{s_l_r})}\text{ and }V_{(\ref{s_l})}=V_{(\ref{s_l_r})}$. Further, by Proposition \ref{s_l_props} (i), one has that global minimizers $(X,Y,S)$ of problems (\ref{s_l}) and (\ref{s_l_r}) own the property (\ref{I-eq}).
\end{proof}

In view of Theorem \ref{assum-rela}, the following conclusion naturally holds.
\begin{corollary}
	For problem (\ref{s_l}), if function $f$ satisfies Assumption \ref{assumpars} (i), then its all global minimizers own the properties (\ref{I-eq}) and (\ref{glb_prp}).
\end{corollary}

For any $I\in[2]^d$, $J\in[2]^{m\times n}$ and $\rho>0$, define $\Theta^{\rho}_I:\mathbb{R}^{r\times d}\rightarrow\mathbb{R}$ and $\tilde\Theta^{\rho}_J:\mathbb{R}^{m\times n}\rightarrow\mathbb{R}$ such that for any $V\in\mathbb{R}^{r\times d}$ and $S\in\mathbb{R}^{m\times n}$,
\begin{align}
	\Theta^{\rho}_{I}(V):=\sum_{i=1}^d\theta_{I_i}(\|V_i\|_2/\rho)\mbox{ and }\tilde\Theta^{\rho}_{J}(S):=\sum_{i=1}^m\sum_{j=1}^n\theta_{J_{ij}}(|S_{ij}|/\rho).\label{Theta-rho}
\end{align}
For any $V\in\mathbb{R}^{r\times d}$, $S\in\mathbb{R}^{m\times n}$ and $\rho>0$, define $I^{\rho,V}\in[2]^d$ and $J^{\rho,S}\in[2]^{m\times n}$ as follows. For any $i\in[d]$,
\begin{equation}\label{I_id}
	I^{\rho,V}_i:=\max\{j\in[2]:\theta_{j}(\|V_i\|_2/\rho)=\theta(\|V_i\|_2/\rho)\}\text{ if }\lambda>0,
\end{equation}
and $I^{\rho,V}_i:=2$ otherwise. For any $(i,j)\in[m,n]$,
\begin{equation}\label{J_id}
	J^{\rho,S}_{ij}:=\max\{l\in[2]:\theta_{l}(|S_{ij}|/\rho)=\theta(|S_{ij}|/\rho)\}\text{ if }\beta>0,
\end{equation}
and $J^{\rho,S}_{ij}:=2$ otherwise.

In what follows, we show that the first-order optimality condition of problem (\ref{s_l_r}) can induce property (\ref{glb_prp}) under Assumption \ref{assumpars}.
\begin{proposition}\label{s_l_r_props}
	Under Assumption \ref{assumpars}, the following statements hold for problem (\ref{s_l_r}).
	\begin{itemize}
		\item[(i)] If $(\bar{X},\bar{Y},\bar{S})\in\mathcal{L}_{(\ref{s_l_r})}$, then 
		\begin{equation}\label{locr-incl}
			\begin{split}
				\exists~~\,&(H_Z,H_S)\in\partial_{(Z,S)}f(\bar{X}\bar{Y}^{\mathbb{T}},\bar{S}),\\
				\emph{s.t.}\,\,&\bm{0}\in H_Z\bar{Y}+\lambda\partial_{X}\Theta^{r}_{I^{r,\bar{X}}}(\bar{X})+N_{\mathcal{B}^X}(\bar{X}),\\
				&\bm{0}\in H_Z^{\mathbb{T}}\bar{X}+\lambda\partial_{Y}\Theta^{r}_{I^{r,\bar{Y}}}(\bar{Y})+N_{\mathcal{B}^Y}(\bar{Y}),\\
				&\bm{0}\in H_S+\beta\partial_{S}\tilde\Theta^{s}_{J^{s,\bar{S}}}(\bar{S})+N_{\mathcal{B}^S}(\bar{S}),
			\end{split}
		\end{equation}
		where $I^{r,\bar{X}},I^{r,\bar{Y}}$ and $J^{s,\bar{S}}$ are fixed index vectors as defined in (\ref{I_id}) and (\ref{J_id}).
		\item[(ii)] Any point $(X,Y,S)\in\mathcal{B}$ satisfying (\ref{locr-incl}) owns the property (\ref{glb_prp}).
	\end{itemize}
\end{proposition}
\begin{proof}
	(i) Let $(\bar{X},\bar{Y},\bar{S})\in\mathcal{L}_{(\ref{s_l_r})}$. Then there exists a $\delta>0$ such that for any $(X,Y,S)\in B_\delta(\bar{X},\bar{Y},\bar{S})\cap\mathcal{B}$, $F_{\theta}^{r,s}(\bar{X},\bar{Y},\bar{S})\leq F_{\theta}^{r,s}(X,Y,S)$.
	Define $\Theta(X,Y,S):=\lambda\big(\Theta^{r}_{I^{r,\bar{X}}}(X)+\Theta^{r}_{I^{r,\bar{Y}}}(Y)\big)+\beta\tilde\Theta^{s}_{I^{s,\bar{S}}}(S),\forall (X,Y,S)\in\mathcal{B}$. Combining with the definitions of $\theta$, $I^{\rho,V}$ and $J^{\rho,S}$ in (\ref{theta_ec}), (\ref{I_id}) and (\ref{J_id}), we have $f(\bar{X}\bar{Y}^{\mathbb{T}},\bar{S})+ \Theta(\bar{X},\bar{Y},\bar{S})=F_{\theta}^{r,s}(\bar{X},\bar{Y},\bar{S})\leq F_{\theta}^{r,s}(X,Y,S)\leq f(XY^{\mathbb{T}},S)+\Theta(X,Y,S),\forall (X,Y,S)\in B_\delta(\bar{X},\bar{Y},\bar{S})\cap\mathcal{B}$. Then, $(\bar{X},\bar{Y},\bar{S})$ is a local minimizer of $f(XY^{\mathbb{T}},S)+\Theta(X,Y,S)$ on $\mathcal{B}$. By \cite[Theorem 6.23]{Penot2013} and \cite[Proposition 10.5, Exercise 10.10]{Rockafellar1998}, it follows from the first-order optimality condition that (\ref{locr-incl}) holds.
	
	(ii) Let $(\bar{X},\bar{Y},\bar{S})$ satisfy (\ref{locr-incl}). Consider that $\lambda>0$ and there exists an $i_0\in[d]$ such that $\|\bar{X}_{i_0}\|_2\in(0,r)$. Then, $I^{r,\bar{X}}_{i_0}=1$. Due to $r<\tau$ in Assumption \ref{assumpars} (ii), we obtain $N_{\mathcal{B}^{X}_{i_0}}(\bar{X}_{i_0})=\{\bm{0}\}$. Combining with Assumption \ref{assumpars} (i), we get that $|\gamma|>L_X\geq\|H_Z\bar{Y}_{i_0}\|_2,\forall (H_Z,H_S)\in\partial_{(Z,S)}f(\bar{X}\bar{Y}^{\mathbb{T}},\bar{S})$ holds  for any $\gamma\in \lambda\partial_{X_{i_0}}\Theta^{r}_{I^{r,\bar{X}}}(\bar{X})+N_{\mathcal{B}^{X}_{i_0}}(\bar{X}_{i_0})$. This yields that (\ref{locr-incl}) does not hold and hence $\|\bar{X}_i\|_2\notin(0,r),\forall i\in[d]$. Similarly, it holds that $\|\bar{Y}_i\|_2\notin(0,r),\forall i\in[d]$ when $\lambda>0$ and $|\bar{S}_{ij}|\notin(0,s),\forall (i,j)\in[m,n]$ when $\beta>0$. Thus, we get (\ref{glb_prp}) with $(X,Y,S)$ replaced by $(\bar{X},\bar{Y},\bar{S})$.
\end{proof}

In view of Proposition \ref{s_l_props} and Proposition \ref{s_l_r_props}, we define two classes of stationary points for problems (\ref{s_l}) and (\ref{s_l_r}). They embed the properties of global minimizers into the first-order optimality conditions of (\ref{s_l}) and (\ref{s_l_r}).
\begin{definition}\label{ssp_defs}
	\item[(i)] For any triple $(X,Y,S)\in\mathbb{R}^{m\times d}\times\mathbb{R}^{n\times d}\times\mathbb{R}^{m\times n}$, the properties (\ref{I-eq}) and (\ref{glb_prp}) are called consistency and isolation, respectively.
	\item[(ii)] $(\bar{X},\bar{Y},\bar{S})$ is called a stationary point of (\ref{s_l_r}), if it satisfies (\ref{locr-incl}).
	\item[(iii)] $(\bar{X},\bar{Y},\bar{S})$ is called a strong stationary point of (\ref{s_l}) or (\ref{s_l_r}), if it is a stationary point of (\ref{s_l}) or (\ref{s_l_r}) with the consistency and isolation properties.\\
	The sets of all strong stationary points of (\ref{s_l}) and (\ref{s_l_r}) are denoted as $\mathcal{S}^s_{(\ref{s_l})}$ and $\mathcal{S}^s_{(\ref{s_l_r})}$, respectively.
\end{definition}

Due to \cite[Proposition 1]{Li2024}, the defined stationary points corresponds to the commonly defined limiting-critical points for (\ref{s_l}) and (\ref{s_l_r}) with a smooth $f$, respectively. Hence, the strong stationary points have stronger optimality condition than the limiting-critical points for (\ref{s_l}) and (\ref{s_l_r}) when $f$ is smooth. Further, we show the relations of strong stationary points and global minimizers between (\ref{s_l}) and (\ref{s_l_r}).
\begin{theorem}\label{glob-stat-eqvs}
	Under Assumption \ref{assumpars}, it holds that
	\begin{equation*}
		\mathcal{G}_{(\ref{s_l_r})}=\mathcal{G}_{(\ref{s_l})}\subseteq\mathcal{S}^s_{(\ref{s_l})}=\mathcal{S}^s_{(\ref{s_l_r})}.
	\end{equation*}
\end{theorem}
\begin{proof}
	Due to Theorem \ref{assum-rela}, one has that $\mathcal{G}_{(\ref{s_l_r})}=\mathcal{G}_{(\ref{s_l})}$. Moreover, by Proposition \ref{s_l_props}, Proposition \ref{s_l_r_props} and Theorem \ref{assum-rela}, we get $\mathcal{G}_{(\ref{s_l})}\subseteq\mathcal{S}^s_{(\ref{s_l})}$ and $\mathcal{G}_{(\ref{s_l_r})}\subseteq\mathcal{S}^s_{(\ref{s_l_r})}$.
	
	Next, it suffices to prove that $\mathcal{S}^s_{(\ref{s_l})}=\mathcal{S}^s_{(\ref{s_l_r})}$.
	
	Suppose $(X,Y,S)\in\mathcal{S}^s_{(\ref{s_l_r})}$.  By (\ref{glb_prp}), one has that $I^{r,X}_i=2,\forall i\in\mathbb{I}_{X}$, $I^{r,Y}_i=2,\forall i\in\mathbb{I}_{Y}$ and $J^{s,S}_{ij}=2,\forall (i,j)\in\mathbb{J}_{S}$, and hence
	\begin{equation}\label{p0-matr}
		[\partial_{X}\Theta^{r}_{I^{r,X}}(X)]_{\mathbb{I}_{X}}=\{\bm{0}\},\,[\partial_{Y}\Theta^{r}_{I^{r,Y}}(Y)]_{\mathbb{I}_{Y}}=\{\bm{0}\},\,[\partial_{S}\tilde\Theta^{s}_{J^{s,S}}(S)]_{\mathbb{J}_{S}}=\{\bm{0}\}.
	\end{equation}
	This together with (\ref{locr-incl}) yields that there exists an $(H_Z,H_S)\in\partial_{(Z,S)}f(XY^{\mathbb{T}},S)$ such that (\ref{loc-incl}) holds. Combining with (\ref{I-eq}) and (\ref{glb_prp}), we get $(X,Y,S)\in\mathcal{S}^s_{(\ref{s_l})}$.
	
	Conversely, let $(X,Y,S)\in\mathcal{S}^s_{(\ref{s_l})}$. Due to (\ref{glb_prp}), we deduce that (\ref{p0-matr}) holds. Further, by (\ref{loc-incl}), we infer that there exists $(H_Z,H_S)\in\partial_{(Z,S)}f(XY^{\mathbb{T}},S)$ such that
	$\bm{0}\in\big[H_ZY+\lambda\partial_{X}\Theta^{r}_{I^{r,X}}(X)+N_{\mathcal{B}^X}(X)\big]_{\mathbb{I}_{X}}$, $\bm{0}\in \big[H_Z^{\mathbb{T}}X+\lambda\partial_{Y}\Theta^{r}_{I^{r,Y}}(Y)+N_{\mathcal{B}^Y}(Y)\big]_{\mathbb{I}_{Y}}$ and $\bm{0}\in \big[H_S+\beta\partial_{S}\tilde\Theta^{s}_{J^{s,S}}(S)+N_{\mathcal{B}^S}(S)\big]_{\mathbb{J}_{S}}$. In view of Assumption \ref{assumpars} and $\bm{0}\in N_{\mathcal{B}^X_{\mathbb{I}^c_{X}}}(X_{\mathbb{I}^c_{X}})=[N_{\mathcal{B}^X}(X)]_{\mathbb{I}^c_{X}}$, we get $\bm{0}\in\big[H_ZY+\lambda\partial_{X}\Theta^{r}_{I^{r,X}}(X)+N_{\mathcal{B}^X}(X)\big]_{\mathbb{I}^c_{X}}$. Similarly, it holds that $\bm{0}\in \big[H_Z^{\mathbb{T}}X+\lambda\partial_{Y}\Theta^{r}_{I^{r,Y}}(Y)+N_{\mathcal{B}^Y}(Y)\big]_{\mathbb{I}^c_{Y}}$ and $\bm{0}\in \big[H_S+\beta\partial_{S}\tilde\Theta^{s}_{J^{s,S}}(S)+N_{\mathcal{B}^S}(S)\big]_{\mathbb{J}^c_{S}}$. Combining with the definitions in (\ref{Theta-rho}), column separability of $\mathcal{B}^X$ and $\mathcal{B}^Y$, and element separability of $\mathcal{B}^S$, we get (\ref{locr-incl}) and hence $(X,Y,S)\in\mathcal{S}^s_{(\ref{s_l_r})}$ by (\ref{I-eq}) and (\ref{glb_prp}).
\end{proof}

\begin{remark}
	Different from using the density function in \cite{bian2024nonsmooth}, mathematical program with equilibrium constraints in \cite{liu2018equivalent} and concave function in \cite{Pan2021} to construct relaxation problems for the $\ell_0$-norm related problems, we only use the lower bound of subdifferential information to characterize the relaxation function for establishing the equivalence between primal and relaxation problems in the sense of global minimizers and strong stationary points.
\end{remark}

In view of Proposition \ref{s_l_props}, Theorem \ref{assum-rela}, Proposition \ref{s_l_r_props} and Theorem \ref{glob-stat-eqvs}, some relations between problems (\ref{s_l}) and (\ref{s_l_r}) are shown as Figure\,\ref{opt-rela}.
\begin{figure}[!t]
	\centering
	\includegraphics[width=2in]{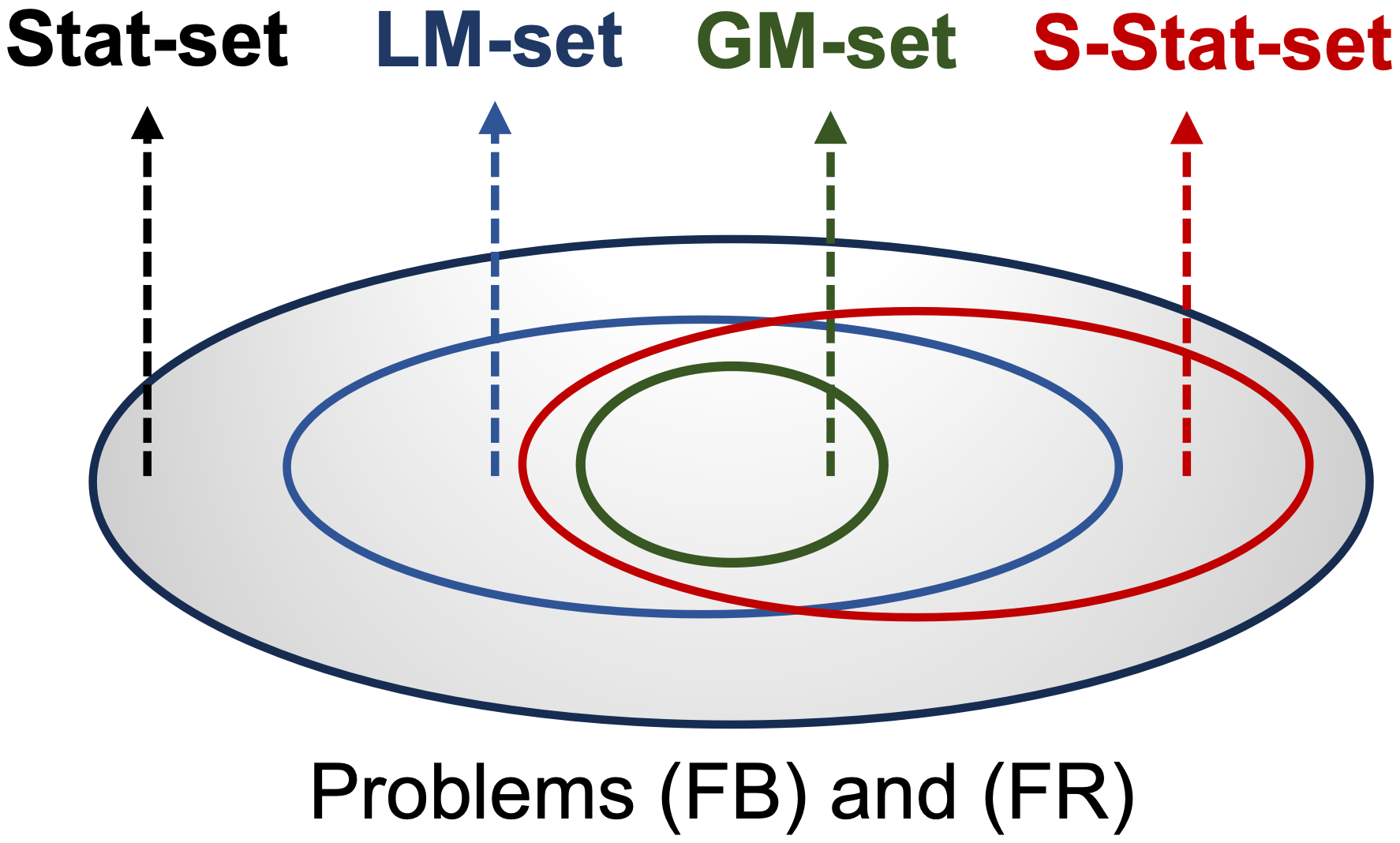}
	\hfil
	\includegraphics[width=3.4in]{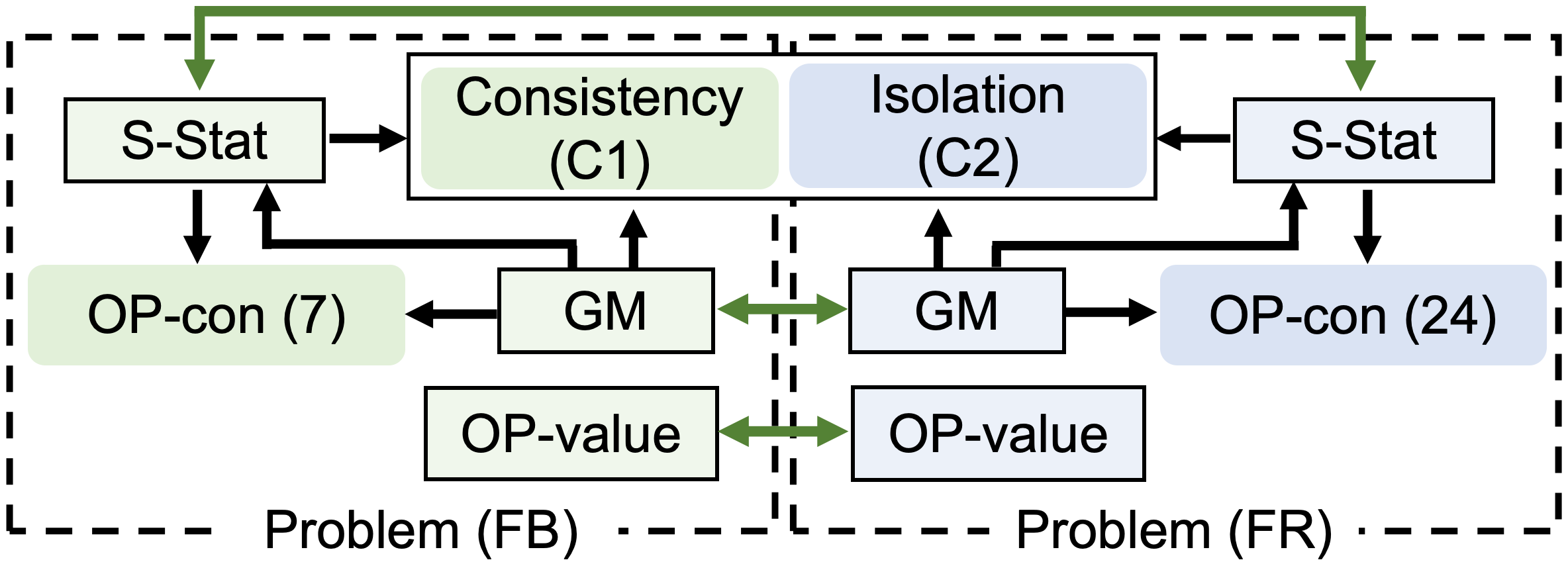}
	\caption{Relation diagram of the optimality conditions for problems (\ref{s_l}) and (\ref{s_l_r}), where `GM/LM-set' means `the set of global/local minimizers', `(S-)Stat-set' means `the set of (strong-)stationary points', and `GM', `OP-value', `OP-con' and `S-Stat' means `global minimizer', `optimal value', `optimal condition' and `strong stationary point', respectively.}
	\label{opt-rela}
\end{figure}

\section{Joint alternating proximal gradient algorithm}\label{section 4}
In this section, we propose a novel proximal gradient-based algorithm to solve the general factorization problem (\ref{GFP}). Throughout this section, we assume that $\inf_{X,Y,S}G(X,Y,S)>-\infty$. Problems (\ref{s_l}) and (\ref{s_l_r}) are special cases of problem (\ref{GFP}). We analyze the convergence of the designed algorithm to the limiting-critical points of problem (\ref{GFP}). Further, we demonstrate that the designed algorithm can find the stationary points of the rank and $\ell_0$ regularized problem (\ref{m_s_l}). Combining with the equivalent relaxation theory in Section \ref{section 3}, we improve this algorithm to show that it can also converge to the strong stationary points of problems (\ref{s_l}) and (\ref{s_l_r}).

\subsection{General factorization case}
In this part, we propose the joint alternating proximal gradient (JA-PG) algorithm for the general factorization problem (\ref{GFP}), analyze its convergence results in the general case, and specific cases on problems (\ref{m_s_l}), (\ref{s_l}) and (\ref{s_l_r}), respectively. Different from the existing popular algorithms that uniformly perform alternating proximal calculations for block-based structures, we treat the factorized matrix variables $X$ and $Y$ simultaneously, and perform alternating proximal calculations with the non-factorized variable $S$ in the designed algorithm. Moreover, compared to the existing alternating proximal gradient-based algorithms, the analysis in this part is conducted under weaker assumption on $f(XY^{\mathbb{T}},S)$, that is, $\nabla f(XY^{\mathbb{T}},S)$ is locally Lipschitz continuous with respect to $(X,Y)$ and $S$, respectively.

In the JA-PG algorithm, for the given $\bar{X}\in\mathbb{R}^{m\times d},\bar{Y}\in\mathbb{R}^{n\times d},\bar{S}\in\mathbb{R}^{m\times n},\iota,\iota_1,\iota_2>0$ and any $X\in\mathbb{R}^{m\times d}$, $Y\in\mathbb{R}^{n\times d}$, $S\in\mathbb{R}^{m\times n}$, define
\begin{align*}
	&f_J(X,Y;\bar{S}):=\,f(XY^{\mathbb{T}},\bar{S})+\lambda\big(u(X)+v(Y)\big),\\
	&f_A(S;\bar{X},\bar{Y}):=\,f(\bar{X}\bar{Y}^{\mathbb{T}},S)+\beta w(S),\\
	&\omega_A(S;\bar{X},\bar{Y},\bar{S},\iota):=\,\beta w(S)+\langle{\nabla}_Sf(\bar{X}{\bar{Y}}^{\mathbb{T}},\bar{S}),S-\bar{S}\rangle+\frac{\iota}{2}\|S-\bar{S}\|_F^2,\\
	&\omega_J(X,Y;\bar{X},\bar{Y},\bar{S},\iota_1,\iota_2):=\,\lambda\big(u(X)+v(Y)\big)+\langle{\nabla}_Zf(\bar{X}{\bar{Y}}^{\mathbb{T}},\bar{S}){\bar{Y}},X-\bar{X}\rangle\\
	&\qquad\qquad\quad+\langle{\nabla}_Zf(\bar{X}{\bar{Y}}^{\mathbb{T}},\bar{S})^{\mathbb{T}}\bar{X},Y-\bar{Y}\rangle+\frac{\iota_1}{2}\|X-\bar{X}\|_F^2+\frac{\iota_2}{2}\|Y-\bar{Y}\|_F^2.
\end{align*}
\begin{algorithm}
	\caption{(\textbf{JA-PG algorithm}): Joint alternating proximal gradient algorithm for problem (\ref{GFP})}
	\label{JA-PG}
	{\bf Initialization:} Choose $(X^0,Y^0,S^0)\in\text{dom}(G)$, $0<\underline{\iota}\leq\overline{\iota}$, $0<c_{\min}\leq c_{\max}$ and $\varrho>1$. 
	Let $\{\iota_{1,k}^B\},\{\iota_{2,k}^B\},\{\iota_{k}^B\}\subseteq[\underline{\iota},\overline{\iota}]$ and $\{c_{1,k}\},\{c_{2,k}\}\subseteq[c_{\min},c_{\max}]$. Set $k=0$.\\
	{\bf Step 1:} Calculate $(X^{k+1},Y^{k+1})$ by JA-J subroutine and go to Step 2.\\
	{\bf Step 2:} Calculate $S^{k+1}$ by JA-A subroutine and go to Step 3.\\
	{\bf Step 3:} Update $k\leftarrow k+1$ and return to Step 1.
\end{algorithm}
\begin{algorithm}[h!]
	\caption{(\textbf{JA-J subroutine}): Calculate $(X^{k+1},Y^{k+1})$ in Step 1 of the JA-PG algorithm}
	\label{JA-PG-step1}
	For $l=1,2,\ldots$, do\\
	(1a) let $\iota_{1,k}=\iota_{1,k}^B\varrho^{l-1}$ and $\iota_{2,k}=\iota_{2,k}^B\varrho^{l-1}$;\\
	(1b) compute $(X^{k+1}(\iota_{1,k}),Y^{k+1}(\iota_{2,k}))$ by
	\vspace{-2mm}\begin{equation}\label{step1-eq}
		(X^{k+1}(\iota_{1,k}),Y^{k+1}(\iota_{2,k}))\in\text{argmin}_{(X,Y)}\,\omega_J(X,Y;X^k,Y^k,S^k,\iota_{1,k},\iota_{2,k});
	\end{equation}
	(1c) if $(X^{k+1}(\iota_{1,k}),Y^{k+1}(\iota_{2,k}))$ satisfies 
	\vspace{-1mm}
	\begin{align*}
		f_J(X^{k+1}(\iota_{1,k}),Y^{k+1}(\iota_{2,k});S^k)
		\leq\,& f_J(X^k,Y^k;S^k)\\&-\frac{c_{1,k}}{2}\|(X^{k+1}(\iota_{1,k}),Y^{k+1}(\iota_{2,k}))-(X^k,Y^k)\|_F^2,
	\end{align*}
	output $(X^{k+1},Y^{k+1})=(X^{k+1}(\iota_{1,k}),Y^{k+1}(\iota_{2,k}))$ and set $(\bar\iota_{1,k},\bar\iota_{2,k})=(\iota_{1,k},\iota_{2,k})$.
\end{algorithm}
\begin{algorithm}[h!]
	\caption{(\textbf{JA-A subroutine}): Calculate $S^{k+1}$ in Step 2 of the JA-PG algorithm}
	\label{JA-PG-step2}
	For $l=1,2,\ldots$, do\\
	(2a) let ${\iota}_{k}={\iota}_{k}^B\varrho^{l-1}$;\\
	(2b) compute $S^{k+1}({\iota}_{k})$ by
	\vspace{-2mm}\begin{equation}\label{step2-eq}
		S^{k+1}({\iota}_{k})\in\text{argmin}_{S}\,\omega_A(S;X^{k+1},Y^{k+1},S^k,\iota_{k});
	\end{equation}
	(2c) if $S^{k+1}(\iota_{k})$ satisfies \vspace{-1mm}$$f_A(S^{k+1}(\iota_{k});X^{k+1},Y^{k+1})\leq f_A(S^k;X^{k+1},Y^{k+1})-\frac{c_{2,k}}{2}\|S^{k+1}(\iota_{k})-S^k\|_F^2,$$
	output $S^{k+1}=S^{k+1}(\iota_{k})$ and set $\bar\iota_{k}=\iota_{k}$.
\end{algorithm}
\begin{remark}\label{prox_cal}
	We show the calculation of subproblems (\ref{step1-eq}) and (\ref{step2-eq}) in the JA-PG algorithm as follows. For any $i\in[d]$,\\
	(i) $X^{k+1}(\iota_{1,k})_i={\rm{prox}}_{\frac{\lambda}{\iota_{1,k}}u}(Q^{k+1}_i)$,
	with $Q^{k+1}={X}^k-\nabla_Z f({X}^k{{Y}^k}^{\mathbb{T}},S^k){Y}^k/\iota_{1,k}$,\\
	(ii)  $Y^{k+1}(\iota_{2,k})_i={\rm{prox}}_{\frac{\lambda}{\iota_{2,k}}v}(\bar{Q}^{k+1}_i)$, with $\bar{Q}^{k+1}={{Y}^k}-\nabla_Z f({X}^k{{Y}^k}^{\mathbb{T}},S^k)^{\mathbb{T}}{X}^k/\iota_{2,k}$.\\
	Moreover, for any $(i,j)\in[m]\times[n]$, $S^{k+1}({\iota}_k)_{ij}={\rm{prox}}_{\frac{\beta}{{\iota}_k}w}(\tilde{Q}^{k+1}_{ij})$,
	with $\tilde{Q}^{k+1}=S^k-\nabla_S f(X^{k+1}{Y^{k+1}}^{\mathbb{T}},S^k)/{\iota}_k$.
\end{remark}

First, we show some basic properties of the JA-PG algorithm.
\begin{proposition}\label{limprop}
	Let $\{(X^k,Y^k,S^k)\}$ be the sequence generated by the JA-PG algorithm for problem (\ref{GFP}). The following statements hold.
	\begin{itemize}
		\item[(i)] The JA-PG algorithm for problem (\ref{GFP}) is well-defined.
		\item[(ii)] $\lim_{k\rightarrow\infty}G(X^k,Y^k,S^k)$ exists and 
		\vspace{-1mm}
		\begin{equation}\label{dif_lim}
			\lim_{k\rightarrow\infty}\|(X^{k+1},Y^{k+1},S^{k+1})-(X^k,Y^k,S^k)\|_F=0.
		\end{equation}
		\item[(iii)] Let $\mathcal{N}\in\mathcal{N}_{\infty}^{\sharp}$ and $W^k:=(X^k,Y^k,S^k),\forall k$. If $\{W^k:k\in\mathcal{N}\}$ is bounded, then there exists a Lipschitz constant $L_f$ of $\nabla_{(X,Y,S)}f(XY^{\mathbb{T}},S)$ on the convex hull of $\{W^k,W^{k+1}:k\in\mathcal{N}\}$ such that
		\vspace{-1mm}\begin{equation}\label{fin_iterLf}
			\bar{\iota}_{1,k},\bar{\iota}_{2,k}\leq\max\big\{\underline\iota,\varrho(c_{1,k}+L_f)\big\},\,
			\bar{\iota}_{k}\leq\max\big\{\underline\iota,\varrho(c_{2,k}+L_f)\big\},\,\forall k\in\mathcal{N},
		\end{equation}
		in the JA-J and JA-A subroutines, and hence $\{\bar{\iota}_{1,k},\bar{\iota}_{2,k},\bar{\iota}_{k}:k\in\mathcal{N}\}$ has an upper bound independent of $k$.
	\end{itemize}
\end{proposition}
\begin{proof}
	(i) Since $u,v$ and $w$ are bounded from below by affine functions, we deduce that $\lim_{\iota_{1,k}\rightarrow+\infty}X^{k+1}(\iota_{1,k})={X}^k$, $\lim_{\iota_{2,k}\rightarrow+\infty}Y^{k+1}(\iota_{2,k})={Y}^k$, $\lim_{\iota_{k}\rightarrow+\infty}S^{k+1}(\iota_{k})$ $=S^k$. Let $\varepsilon>0$ be any given. Then there exist $\underline{\iota}_{1,k},\underline{\iota}_{2,k}\geq0$ such that for any $\iota_1,\iota_2\geq\underline{\iota}_{1,k}$ and $\iota\geq\underline{\iota}_{2,k}$, $X^{k+1}(\iota_1)\in B_\varepsilon({X}^k),\,\,Y^{k+1}(\iota_2)\in B_\varepsilon({Y}^k),\,\,S^{k+1}(\iota)\in B_\varepsilon(S^k)$.
	Since $\nabla f$ is locally Lipschitz continuous, we deduce that there exist Lipschitz constants $L_{1,k}$ and $L_{2,k}$ of $\nabla_{(X,Y)}f(XY^{\mathbb{T}},S)$ on $B_\varepsilon({X}^k)\times B_\varepsilon({Y}^k)$ and $\nabla_{S}f(XY^{\mathbb{T}},S)$ on $B_\varepsilon(S^k)$, respectively.
	
	Next, let $\iota_{1,k},\iota_{2,k}\geq\max\{\underline{\iota}_{1,k},c_{1,k}+L_{1,k}\}$. Similar to \cite[Lemma 2]{Bolte2014Proximal}, we have that
	\begin{equation}\label{ieq_XY}
		\begin{split}
			f_J(X^k,Y^k;S^k)
			\geq\,&f_J(X^{k+1}(\iota_{1,k}),Y^{k+1}(\iota_{2,k});S^k)\\&+\frac{c_{1,k}}{2}\|(X^{k+1}(\iota_{1,k}),Y^{k+1}(\iota_{2,k}))-({X}^k,{Y}^k)\|_F^2,\,\forall k.
		\end{split}
	\end{equation}
	Similarly, let $\iota_{k}\geq\max\{\underline{\iota}_{2,k},c_{2,k}+L_{2,k}\}$, we deduce that
	\begin{equation}\label{ieq_S}
		f_A(S^k;X^{k+1},Y^{k+1})\geq f_A(S^{k+1}({\iota}_{k});X^{k+1},Y^{k+1})+\frac{{c}_{2,k}}{2}\|S^{k+1}({\iota}_{k})-S^k\|_F^2,\,\forall k.
	\end{equation}
	These yield that (1c) and (2c) hold and hence the statement holds.
	
	(ii) Combining with (\ref{ieq_XY}) and (\ref{ieq_S}), we have that
	\begin{equation}\label{limxk}
		\begin{split}
			&G(X^k,Y^k,S^k)-G(X^{k+1},Y^{k+1},S^{k+1})\\
			=\,&f_J(X^k,Y^k;S^k)-f_J(X^{k+1},Y^{k+1};S^k)\\
			&+f_A(S^k;X^{k+1},Y^{k+1})-f_A(S^{k+1};X^{k+1},Y^{k+1})\\
			\geq\,&\frac{c_{\min}}{2}\|(X^{k+1},Y^{k+1},S^{k+1})-({X}^k,{Y}^k,S^k)\|_F^2,\,\forall k.
		\end{split}
	\end{equation}
	This means that $G(X^k,Y^k,S^k)$ is non-increasing with respect to $k$. Since $G$ is bounded from below, $\lim_{k\rightarrow\infty}G(X^k,Y^k,S^k)$ exists. Further, by (\ref{limxk}) and $c_{\min}>0$, we infer that (\ref{dif_lim}) holds.
	
	(iii) Let $W^k:=(X^k,Y^k,S^k),\forall k$. Consider that $\{W^k:k\in\mathcal{N}\}$ is bounded. By (\ref{dif_lim}), we deduce that $\{W^{k+1}:k\in\mathcal{N}\}$ is also bounded. Denote the convex hull of $\{W^k,W^{k+1}:k\in\mathcal{N}\}$ as $\mathcal{C}$. Then, $\mathcal{C}$ is bounded. Based on the locally Lipschitz continuity of $\nabla f$, there exist a Lipschitz constant $L_f$ of $\nabla_{(X,Y,S)}f(XY^{\mathbb{T}},S)$ on $\mathcal{C}$. Similar to the proof of (i), we deduce that $\underline{\iota}_{1,k}=\underline{\iota}_{2,k}=0$ and $L_{1,k}=L_{2,k}=L_f$ for any $k\in\mathcal{N}$. Combining with (\ref{ieq_XY}), (\ref{ieq_S}), $\iota_{1,k}^B,\iota_{2,k}^B,\iota_k^B\in[\underline{\iota},\overline{\iota}]$ and $\varrho>1$, we obtain (\ref{fin_iterLf}). This together with $c_{1,k},c_{2,k}\leq c_{\max}$ yields the boundedness of $\{\bar{\iota}_{1,k},\bar{\iota}_{2,k},\bar{\iota}_{k}:k\in\mathcal{N}\}$.
\end{proof}

In what follows, we show the convergence of the JA-PG algorithm for (\ref{GFP}).
\begin{theorem}\label{GF_thm}
	Let $\{(X^k,Y^k,S^k)\}$ be the sequence generated by the JA-PG algorithm for problem (\ref{GFP}). Then any accumulation point $(X^*,Y^*,S^*)$ of $\{(X^k,Y^k,S^k)\}$ is a limiting-critical point of problem (\ref{GFP}), that is $\bm{0}\in\partial G(X^*,Y^*,S^*)$.
\end{theorem}
\begin{proof}
	Let $(X^*,Y^*,S^*)$ be an accumulation point of $\{(X^k,Y^k,S^k)\}$. Then there exists a subsequence $\{k_l\}$ such that
	\begin{equation}\label{accum}
		\lim_{k_l\rightarrow\infty}(X^{k_l},Y^{k_l},S^{k_l})=(X^*,Y^*,S^*).
	\end{equation}
	By (\ref{dif_lim}), we deduce that $\lim_{k_l\rightarrow\infty}(X^{k_l-1},Y^{k_l-1},S^{k_l-1})=(X^*,Y^*,S^*)$. As a result, $\{(X^{k_l},Y^{k_l},S^{k_l}),(X^{k_l-1},$ $Y^{k_l-1},S^{k_l-1})\}$ is bounded. By Proposition \ref{limprop}(ii) and (\ref{accum}), 
	\begin{equation}\label{seq_lim}
		\lim_{k_l\rightarrow\infty}G(X^{k_l},Y^{k_l},S^{k_l})=G(X^*,Y^*,S^*).
	\end{equation}
	
	It follows from the optimality condition of (\ref{step1-eq}) and (\ref{step2-eq}) with $\iota_{i,k_l}=\bar\iota_{i,k_l},i\in[2]$ and $\iota_{k_l}=\bar\iota_{k_l}$ that there exists $\xi^{k_l}\in\partial_{(X,Y)}\lambda(u(X^{k_l})+v(Y^{k_l}))$ and $\zeta^{k_l}\in\partial_{S}\beta w(S^{k_l})$ such that	
	\begin{equation*}
		\begin{split}
			&\bm{0}=\eta^{k_l-1}+\xi^{k_l}+\big(\bar\iota_{1,k_l-1}(X^{k_l}-X^{k_l-1}),\bar\iota_{2,k_l-1}(Y^{k_l}-Y^{k_l-1})\big),\\
			&\bm{0}=\nabla_S f(X^{k_l}{Y^{k_l}}^{\mathbb{T}},S^{k_l-1})+\zeta^{k_l}+\bar\iota_{k_l-1}(S^{k_l}-S^{k_l-1}),
		\end{split}
	\end{equation*}
	where $\eta^{t}:=(\nabla_Z f(X^{t}{Y^{t}}^{\mathbb{T}},S^{t}){Y^{t}},\nabla_Z f(X^{t}{Y^{t}}^{\mathbb{T}},S^{t})^{\mathbb{T}}{X^{t}}),\forall t\in\mathbb{N}$. Combining with $\eta^{k_l}+\xi^{k_l}\in\partial_{(X,Y)}G(X^{k_l},$ $Y^{k_l},S^{k_l})$, $\nabla_S f(X^{k_l}{Y^{k_l}}^{\mathbb{T}},S^{k_l})+\zeta^{k_l}\in\partial_{S}G(X^{k_l},Y^{k_l},S^{k_l})$ and \cite[Proposition 1]{Bolte2014Proximal}, we obtain that
	\begin{equation}\label{subd_fo}
		(\kappa_1^{k_l},\kappa_2^{k_l})\in\partial_{(X,Y)}G(X^{k_l},Y^{k_l},S^{k_l})\times\partial_{S}G(X^{k_l},Y^{k_l},S^{k_l})=\partial G(X^{k_l},Y^{k_l},S^{k_l}),
	\end{equation}
	where $\kappa_1^{k_l}:=\eta^{k_l}-\eta^{k_l-1}+\big(\bar\iota_{1,k_l-1}(X^{k_l-1}-X^{k_l}),\bar\iota_{2,k_l-1}(Y^{k_l-1}-Y^{k_l})\big)$ and $\kappa_2^{k_l}:=\nabla_S f(X^{k_l}{Y^{k_l}}^{\mathbb{T}},S^{k_l})-\nabla_S f(X^{k_l}{Y^{k_l}}^{\mathbb{T}},S^{k_l-1})+\bar\iota_{k_l-1}$ $(S^{k_l-1}-S^{k_l})$. By the locally Lipschitz continuity of $\nabla f$, the boundedness of $\{(X^{k_l},Y^{k_l},S^{k_l}),(X^{k_l-1},Y^{k_l-1},S^{k_l-1})\}$ and Proposition \ref{limprop} (iii), there exists an $L\geq0$ such that $\|(\kappa_1^{k_l},\kappa_2^{k_l})\|_F\leq L\|(X^{k_l},Y^{k_l},$ $S^{k_l})-(X^{k_l-1},Y^{k_l-1},S^{k_l-1})\|_F$.
	This together with (\ref{dif_lim}) yields that
	\begin{equation}\label{lim_ze}
		\lim_{k_l\rightarrow\infty}(\kappa_1^{k_l},\kappa_2^{k_l})=\bm{0}.
	\end{equation}
	
	Combining (\ref{accum}), (\ref{seq_lim}), (\ref{subd_fo}) and (\ref{lim_ze}), by the closedness property of $\partial G$ in \cite[Remark\,1\,(ii)]{Bolte2014Proximal}, it holds that $\bm{0}\in\partial G(X^*,Y^*,S^*)$.
\end{proof}

In view of Theorem \ref{GF_thm}, Definition \ref{ssp_defs} and \cite[Proposition 10.5]{Rockafellar1998}, the following statement naturally holds for the relaxation problem (\ref{s_l_r}).
\begin{corollary}
	Let $\{(X^k,Y^k,S^k)\}$ be the sequence generated by the JA-PG algorithm for problem (\ref{s_l_r}). Then any accumulation point of $\{(X^k,Y^k,S^k)\}$ is a stationary point of problem (\ref{s_l_r}).
\end{corollary}

For the rank and $\ell_0$ regularized problem (\ref{m_s_l}) and factorization problem (\ref{s_l}), we have the following convergence results of the JA-PG algorithm.
\begin{corollary}
	Let $\{(X^k,Y^k,S^k)\}$ be the sequence generated by the JA-PG algorithm for problem (\ref{s_l}). Then any accumulation point of $\{(X^k,Y^k,S^k)\}$ is a stationary point of problem (\ref{s_l}). Moreover, when $\lambda>0$ and $\tau=+\infty$ in problem (\ref{s_l}), any accumulation point of $\{(X^k,Y^k,S^k)\}$ is a stationary point of problem (\ref{m_s_l}).
\end{corollary}
\begin{proof}
	Let $(X^*,Y^*,S^*)$ be an accumulation point of $(X^k,Y^k,S^k)$. Denote $\Gamma:=\mathbb{R}^{m\times d}\times\mathbb{R}^{n\times d}\times\mathbb{R}^{m\times n}$. Define $h:\Gamma\rightarrow\overline{\mathbb{R}}$ with $h(X,Y,S)=\lambda\big(\text{nnzc}(X)+\text{nnzc}(Y)\big)+\beta\|S\|_0+\delta_{\mathcal{B}}(X,Y,S),\forall (X,Y,S)\in\Gamma$.
	Similar to the proof of \cite[Proposition 1]{Li2024}, we deduce that $\partial h(\bar{X},\bar{Y},\bar{S})=N_{{\underline{\mathcal{B}}}^X_{\mathbb{I}_{\bar{X}}}}(\bar{X})\times N_{{\underline{\mathcal{B}}}^Y_{\mathbb{I}_{\bar{Y}}}}(\bar{Y})\times N_{{\underline{\mathcal{B}}}^S_{\mathbb{J}_{\bar{S}}}}(\bar{S}),\forall(\bar{X},\bar{Y},\bar{S})\in\Gamma$. Combining with the continuously differentiability of $f$ and Theorem \ref{GF_thm}, we have that $(X^*,Y^*,S^*)$ is a stationary point of problem (\ref{s_l}). Further, when $\lambda>0$ and $\tau=+\infty$ in problem (\ref{s_l}), by Definition \ref{sd_defs} and Theorem \ref{md_eq_lc_sd} (ii), we have that $(X^*{Y^*}^{\mathbb{T}},S^*)$ is a stationary point of problem (\ref{m_s_l}).
\end{proof}

\subsection{Specific factorization case}
In this part, we propose an adaptive JA-PG algorithm, named by AJA-PG algorithm, to solve the relaxation problem (\ref{s_l_r}), and analyze its convergence to the strong stationary points of problems (\ref{s_l}) and (\ref{s_l_r}). The main improvement lies in updating the subproblems to be solved based on the iterate generated in the previous iteration, and automatically reducing the dimensionality of the subproblem corresponding to matrix factorization.
\medskip

In the AJA-PG algorithm, for the given $\bar{X}\in\mathbb{R}^{m\times d},\bar{Y}\in\mathbb{R}^{n\times d},\bar{S}\in\mathbb{R}^{m\times n},\iota,\iota_1,\iota_2,$ $\rho>0,I\in[2]^{m\times n},I^1,I^2\in[2]^d$ and any $X\in\mathbb{R}^{m\times d}$, $Y\in\mathbb{R}^{n\times d}$, $S\in\mathbb{R}^{m\times n}$, define
\begin{align*}
	&f_J^\rho(X,Y;\bar{S}):=\,f(XY^{\mathbb{T}},\bar{S})+\lambda\big(\Theta^{\rho}(X)+\Theta^{\rho}(Y)\big),\\
	&f_A^\rho(S;\bar{X},\bar{Y}):=\,f(\bar{X}\bar{Y}^{\mathbb{T}},S)+\beta\tilde\Theta^{\rho}(S),\\
	&\omega^{\rho,I}_A(S;\bar{X},\bar{Y},\bar{S},\iota):=\,\beta\tilde\Theta^{\rho}_{I}(S)+\langle{\nabla}_Sf(\bar{X}{\bar{Y}}^{\mathbb{T}},\bar{S}),S-\bar{S}\rangle+\frac{\iota}{2}\|S-\bar{S}\|_F^2+\delta_{\mathcal{B}^S}(S),\\
	&\omega^{\rho,I^1,I^2}_J(X,Y;\bar{X},\bar{Y},\bar{S},\iota_1,\iota_2):=\,\lambda\big(\Theta^{\rho}_{I^1}(X)+\Theta^{\rho}_{I^2}(Y)\big)+\langle{\nabla}_Zf(\bar{X}{\bar{Y}}^{\mathbb{T}},\bar{S}){\bar{Y}},X-\bar{X}\rangle\\
	&+\langle{\nabla}_Zf(\bar{X}{\bar{Y}}^{\mathbb{T}},\bar{S})^{\mathbb{T}}\bar{X},Y-\bar{Y}\rangle+\frac{\iota_1}{2}\|X-\bar{X}\|_F^2+\frac{\iota_2}{2}\|Y-\bar{Y}\|_F^2+\delta_{\mathcal{B}^X}(X)+\delta_{\mathcal{B}^Y}(Y).
\end{align*}

\begin{algorithm}
	\caption{(\textbf{AJA-PG algorithm}): Adaptive joint alternating proximal gradient algorithm for (\ref{s_l_r})}
	\label{AJA-PG}
	
	{\bf Initialization:} Choose $(X^0,Y^0,S^0)\in\mathcal{B}$, $0<\underline{\iota}\leq\overline{\iota}$, $0<c_{\min}\leq c_{\max}$, $\varrho>1$ and $K,\tilde{K}\in\mathbb{N}$. Let $\{{r}_k\}$ and $\{{{s}_k}\}$ be positive sequences satisfying 
	\begin{equation}\label{choice_rho}
		r_k=r,\forall k\geq K\text{ and }{s_k}=s,\forall k\geq \tilde{K}.
	\end{equation}
	Let $\{\iota_{1,k}^B\},\{\iota_{2,k}^B\},\{\iota_{k}^B\}\subseteq[\underline{\iota},\overline{\iota}]$ and $\{c_{1,k}\},\{c_{2,k}\}\subseteq[c_{\min},c_{\max}]$. For any $k\geq0$, denote $\mathbb{I}_k:=\mathbb{I}_{X^k}\cap\mathbb{I}_{Y^k}$, $I^{1,k}=I^{{r}_k,\underline{X}^k_{\mathbb{I}_k}}$, $I^{2,k}=I^{{r}_k,\underline{Y}^k_{\mathbb{I}_k}}$ and $J^k=J^{{s}_k,S^k}$. Set $k=0$.\\
	{\bf Step 1:} Calculate $(X^{k+1},Y^{k+1})$ by JA-J subroutine, with $\omega_J$, $f_J$ and $(X^k,Y^k)$ replaced by $\omega^{r_k,I^{1,k},I^{2,k}}_J$, $f_J^{r_k}$ and $(\underline{X}^k_{\mathbb{I}_k},\underline{Y}^k_{\mathbb{I}_k})$, respectively, and go to Step 2.\\
	{\bf Step 2:} Calculate $S^{k+1}$ by JA-A subroutine, with $\omega_A$ and $f_A$ replaced by $\omega_A^{s_k,J^k}$ and $f_A^{s_k}$, respectively, and go to Step 3.\\
	{\bf Step 3:} Update $k\leftarrow k+1$ and return to Step 1.
\end{algorithm}

Similar to the analysis of Proposition \ref{limprop}, by \cite[Proposition 9]{Li2024}, the statements of Proposition \ref{limprop} are applicable to the AJA-PG algorithm for problem (\ref{s_l_r}), where $G$ and $(X^k,Y^k)$ are replaced by $F_\theta^{r,s}$ and $(\underline{X}^k_{\mathbb{I}_k},\underline{Y}^k_{\mathbb{I}_k})$, respectively.
\medskip

In the subsequent analysis, we will use the following assumption.
\begin{assumption}\label{furpar}
	\item[(i)] $\theta_1$ defined by (\ref{theta1_def}) is locally Lipschitz continuous, and $$\bar{\eta}:=\inf\{\eta\in \partial^c\theta_{1}(t):t\in\mathbb{R}_+\}>0.$$
	\item[(ii)] Assumption \ref{assumpars} (i) holds, parameter $\tilde{\eta}$ is chosen from $(0,\bar{\eta}]$, and parameters $r$ and $s$ in problem (\ref{s_l_r}) are chosen by
	\begin{equation}\label{cap_relas_f}
		\begin{split}
			&0<r<\min\big\{\min\{1,\tilde{\eta}\}\cdotp\lambda/\max\{L_X,L_Y\},\tau\big\}\,\,\text{if}\,\,\lambda>0,\\
			&0<s<\min\{1,\tilde{\eta}\}\cdotp\beta/L_S\,\,\text{if}\,\,\beta>0.
		\end{split}
	\end{equation}
\end{assumption}

Obviously, Assumption \ref{furpar} (i) implies the non-decreasing property of $\theta_1$. In view of Assumption \ref{furpar} (i), we have the following computational inferences for its proximal operator, which plays a key role in the subsequent analysis.
\begin{proposition}\label{prox_0}
	Define function $\bar\theta_1:\mathbb{R}_+\rightarrow\mathbb{R}_+$ with $\bar\theta_1(t)=t,\forall t\in\mathbb{R}_+$. Suppose that Assumption \ref{furpar} (i) holds, and constants $\rho,\,\bar{\rho}$ satisfy $0<\rho\leq\min\{\tilde{\eta},1\}\,\bar{\rho}$. For any given $\gamma,\varsigma>0$ and $l,u\in\mathbb{R}^l$ with $l\leq\bm{0}\leq u$, it holds that for any $z\in\mathbb{R}^l$,
	\begin{itemize}
		\item if $\|z\|_1\leq\gamma/{\bar\rho}$, ${\rm{prox}}_{\gamma\theta_{1}(\|\cdotp\|_1/\rho)+\delta_{[l,u]}(\cdotp)}(z)={\rm{prox}}_{\gamma\bar\theta_{1}(\|\cdotp\|_1/{\bar\rho})+\delta_{[l,u]}(\cdotp)}(z)=\bm{0}$;
		\item if $\|z\|_2\leq\gamma/{\bar\rho}$, ${\rm{prox}}_{\gamma\theta_{1}(\|\cdotp\|_2/\rho)+\delta_{B_{\varsigma}}(\cdotp)}(z)={\rm{prox}}_{\gamma\bar\theta_{1}(\|\cdotp\|_2/{\bar\rho})+\delta_{B_{\varsigma}}(\cdotp)}(z)=\bm{0}$.
	\end{itemize}
\end{proposition}
\begin{proof}
	According to Assumption \ref{furpar} (i) and $\bar\theta_1(t)=t,\,\forall t\in[0,+\infty)$, one has that $\theta_1(t/\rho)\geq \bar\theta_1(t/\bar\rho)$ when $t\geq0$, ${\rm{prox}}_{\gamma\bar\theta_{1}(\|\cdotp\|_1/{\bar\rho})+\delta_{[l,u]}(\cdotp)}(z)=\bm{0}$ when $\|z\|_1\leq\gamma/{\bar\rho}$, and ${\rm{prox}}_{\gamma\bar\theta_{1}(\|\cdotp\|_2/{\bar\rho})+\delta_{B_{\varsigma}}(\cdotp)}(z)$ $=\bm{0}$ when $\|z\|_2\leq\gamma/{\bar\rho}$. Further, by the non-decreasing property of $\theta_1$ and $\rho\leq\bar\rho$, we have $\theta_{1}(\|z\|_q/\rho)\geq\theta_{1}(\|z\|_q/\bar\rho)\geq\bar\theta_{1}(\|z\|_q/\bar\rho),\,\forall q\in\{1,2\}$. Combining with $\theta_{1}(0)=\bar\theta_{1}(0)$, we get the statement.
\end{proof}

Next, under the bounded assumption on the feasible region of problem (\ref{s_l}), we will present the convergence analysis concerning the indicator sets, iteration sequence, and their corresponding function values generated by the AJA-PG algorithm.
\begin{theorem}\label{JA-PG-conv-thm}
	If Assumption \ref{furpar} holds and $\mathcal{B}$ is bounded, then the following properties hold for the sequence $\{(X^k,Y^k,S^k)\}$ generated by the AJA-PG algorithm.
	\begin{itemize}
		\item[(i)] $\mathbb{I}_{X^{k+1}}\subseteq\mathbb{I}_k\subseteq\mathbb{I}_{X^k}$, $\mathbb{I}_{Y^{k+1}}\subseteq\mathbb{I}_k\subseteq\mathbb{I}_{Y^k},\forall k\geq0$ and $\mathbb{J}_{S^{k+1}}\subseteq\mathbb{J}_{S^k},\forall k\geq \tilde{K}$.
		\item[(ii)] There exist $\hat{K}\geq \tilde{K}$, $\hat{\mathbb{I}}\subseteq[d]$ and $\hat{\mathbb{J}}\subseteq[m,n]$ such that $\mathbb{I}_{X^k}=\mathbb{I}_{Y^k}=\hat{\mathbb{I}}$ and $\mathbb{J}_{S^k}=\hat{\mathbb{J}},\forall k\geq\hat{K}$, and hence
		\vspace{-1mm}\begin{equation}\label{var-ret}
			\underline{X}^k_{\mathbb{I}_k}=X^k\text{ and }\underline{Y}^k_{\mathbb{I}_k}=Y^k,\forall k\geq\hat{K}.
		\end{equation}
		\item[(iii)] There exists $\check{K}\geq\hat{K}$ such that $(X^k,Y^k,S^k)$ own the consistency and isolation properties (\ref{I-eq}) and (\ref{glb_prp}) for any $k\geq\check{K}$.
		\item[(iv)] $\lim_{k\rightarrow\infty}F(X^k,Y^k,S^k)$ exists.
		\item[(v)] Any accumulation point of $\{(X^k,Y^k,S^k)\}$ is a strong stationary point of problems (\ref{s_l}) and (\ref{s_l_r}).
		\item[(vi)] If $f(XY^{\mathbb{T}},S)$ is semialgebraic, then $\lim_{k\rightarrow\infty}(X^k,Y^k,S^k)=(X^*,Y^*,S^*)$ and its convergence rate is R-sublinear convergence at worst, i.e.
		\vspace{-1mm}\begin{equation*}
			\exists~\alpha\in({1}/{2},1)\emph{ and }c>0,
			\emph{ s.t. }\|(X^k,Y^k,S^k)-(X^*,Y^*,S^*)\|_F\leq ck^{-\frac{1-\alpha}{2\alpha-1}},\,\forall k.
		\end{equation*}
	\end{itemize}
\end{theorem}
\begin{proof}
	(i) In view of the boundedness of $\mathcal{B}$ and the locally Lipschitz continuity of $\nabla f$, Assumption \ref{assumpars} (i) naturally holds. In view of Remark \ref{prox_cal} and (\ref{cap_relas_f}), one has $Q^{k+1}_{i}=\bm{0},\bar{Q}^{k+1}_{i}=\bm{0},\forall i\notin\mathbb{I}_k$ and when $k\geq \tilde{K}$, $|\tilde{Q}^{k+1}_{ij}|\leq L_S/{\iota}_k<\beta/({\iota}_k s),\forall (i,j)\notin \mathbb{J}_{S^k}$. Then, by Proposition \ref{prox_0}, it holds that 
	\begin{align}
		&X^{k+1}_i=Y^{k+1}_i=\bm{0}\text{ if }i\notin \mathbb{I}_k,\,\,\forall k\geq0,\label{IXY_incl}\\
		&S^{k+1}_{ij}=0\text{ if }(i,j)\notin \mathbb{J}_{S^k},\,\,\forall k\geq \tilde{K},\label{JS_incl}
	\end{align}
	and hence $\mathbb{I}_{X_{k+1}}\subseteq\mathbb{I}_k,\forall k\geq0$ and $\mathbb{J}_{S^{k+1}}\subseteq\mathbb{J}_{S^k},\forall k\geq \tilde{K}$. Combining with $\mathbb{I}_k=\mathbb{I}_{X^k}\cap\mathbb{I}_{Y^k}$, we obtain that $\mathbb{I}_{X^{k+1}}\subseteq\mathbb{I}_k\subseteq\mathbb{I}_{X^k}$ and $\mathbb{I}_{Y^{k+1}}\subseteq\mathbb{I}_k\subseteq\mathbb{I}_{Y^k},\forall k\geq0$.
	
	(ii) By (i) and $\mathbb{I}_k=\mathbb{I}_{X^k}\cap\mathbb{I}_{Y^k}$, we have $\mathbb{I}_{k+1}\subseteq\mathbb{I}_{X^{k+1}}\subseteq\mathbb{I}_k\subseteq[d]$ and $\mathbb{I}_{k+1}\subseteq\mathbb{I}_{Y^{k+1}}\subseteq\mathbb{I}_k\subseteq[d]$. It implies that $\lim_{k\rightarrow\infty}\mathbb{I}_k$ exists and hence $\lim_{k\rightarrow\infty}\mathbb{I}_{X^k}$ and $\lim_{k\rightarrow\infty}\mathbb{I}_{Y^k}$ exist. Moreover, due to $\mathbb{J}_{S^{k+1}}\subseteq\mathbb{J}_{S^k}\subseteq[m,n],\forall k\geq\tilde{K}$ in (i), we obtain the existence of $\lim_{k\rightarrow\infty}\mathbb{J}_{S^k}$. As a result, the statement naturally holds by (i). Due to $\mathbb{I}_{X^k}=\mathbb{I}_{Y^k},\forall k\geq\hat{K}$, we get (\ref{var-ret}).
	
	(iii) According to $\mathbb{I}_{X^k}=\mathbb{I}_{Y^k},\forall k\geq\hat{K}$ in (ii), we get that $(X^k,Y^k,S^k),k\geq\hat{K}$ own the consistency property (\ref{I-eq}).
	
	Next, we prove that $(X^k,Y^k,S^k)$ owns the isolation property (\ref{glb_prp}) after finite iterations.
	
	By the boundedness of $\mathcal{B}$, there exist positive $\varsigma$ and $ L_f$ such that $\|(X,Y,S)\|^2_F\leq\varsigma,\forall (X,Y,S)\in\mathcal{B}$ and $\nabla_{(X,Y,S)}f(XY^{\mathbb{T}},S)$ is Lipschitz continuous on $\mathcal{B}$ with modulus $L_f$. Due to (\ref{var-ret}) and Proposition \ref{limprop} (ii), we get $\lim_{k\rightarrow+\infty}\|(X^{k+1},Y^{k+1},S^{k+1})-(X^k,Y^k,S^k)\|_F=0$. Then there exists a $\bar{K}\geq \max\{\hat{K},K\}$ such that for any $k\geq\bar{K}$, one has that for $ i\in[d]$,
	\begin{equation}\label{distXY}
		\max\{\|X^{k+1}_i-X^k_i\|_2,\|Y^{k+1}_i-Y^k_i\|_2\}\\
		<({\lambda}\tilde{\eta}/{r}-\max\{L_X,L_Y\})/\max\{\underline{\iota},\varrho(c_{\max}+ L_f)\},
	\end{equation}
	and for $(i,j)\in[m,n]$,
	\begin{equation}\label{distS}
		|S^{k+1}_{ij}-S^k_{ij}|<({\beta}\tilde{\eta}/{s}-L_S)/\max\{\underline{\iota},\varrho (c_{\max}+L_f)\},
	\end{equation}
	where $r$ and $s$ are given in (\ref{choice_rho}). Moreover, it follows from (\ref{var-ret}) and the optimality conditions of (\ref{step1-eq}) with $\iota_{i,k}=\bar\iota_{i,k},i\in[2]$ that
	\begin{align}\label{Xincl}
		\bm{0}\in\,\nabla_Z f(X^k{Y^k}^{\mathbb{T}},S^k){Y^k}+\bar\iota_{1,k}(X^{k+1}-X^k)+\lambda\partial_X\Theta^{r}_{I^{1,k}}(X^{k+1})+N_{\mathcal{B}^X}(X^{k+1}).
	\end{align}
	Consider $\lambda>0$ and suppose that
	\begin{equation}\label{cs1c}
		\exists~\bar{k}\geq\bar{K}\text{ and }\bar{j}\in[d],\,\text{ s.t. }\|X^{\bar{k}}_{\bar{j}}\|_2<r.
	\end{equation}
	Then $I^{1,\bar{k}}_{\bar{j}}=1$. This together with $r<\tau$ in Assumption \ref{furpar} (ii) yields that if $X^{\bar{k}+1}_{\bar{j}}\neq\bm{0}$, then $\|\bar\gamma\|_2>\lambda\tilde{\eta}/{r},\forall\bar\gamma\in \lambda\partial_{X_{\bar{j}}}\Theta^{r}_{I^{1,k}}(X^{\bar{k}+1})+N_{\mathcal{B}^{X}_{\bar{j}}}(X^{\bar{k}+1}_{\bar{j}})$. Moreover, combining Assumption \ref{furpar} (ii), Proposition \ref{limprop} (iii) and (\ref{distXY}), we obtain that $\|\nabla_Zf(X^{\bar{k}}{Y^{\bar{k}}}^{\mathbb{T}},S^{\bar{k}}){Y^{\bar{k}}_{\bar{j}}}+\bar\iota_{1,\bar{k}}(X^{\bar{k}+1}_{\bar{j}}-X^{\bar{k}}_{\bar{j}})\|_2<\lambda\tilde{\eta}/{r}$. Then we have the contradiction to (\ref{Xincl}) if $X^{\bar{k}+1}_{\bar{j}}\neq\bm{0}$. As a result, we obtain $X^{\bar{k}+1}_{\bar{j}}=\bm{0}$ when (\ref{cs1c}) holds, which implies that there exists $\check{K}\geq\hat{K}$ such that $\|X^k_j\|_2\notin(0,r),\forall j\in[d]$ for all $k\geq\check{K}$ when $\lambda>0$. Similarly, we have the corresponding results for $Y^k$ and $S^k$ as in statement (iii).
	
	(iv) According to (iii), it holds that $F_{\theta}^{r,s}(X^k,Y^k,S^k)=F(X^k,Y^k,S^k),\forall k\geq\check{K}$. Then, by Proposition \ref{limprop} (ii), we obtain the existence of $\lim_{k\rightarrow+\infty}F(X^k,Y^k,S^k)$.
	
	(v) By (ii) and (iii), there exist $I\in[2]^d$ and $J\in[2]^{m\times n}$ such that $I^{1,k}=I^{2,k}=I,J^k=J,\forall k\geq\check{K}$. Combining with Theorems \ref{glob-stat-eqvs} and \ref{GF_thm}, this statement holds.
	
	(vi) Define $\hat{f}(X,Y,S):=f(XY^{\mathbb{T}},S)+\delta_{\underline{\mathcal{B}}_{\hat{\mathbb{I}}}^X}(X)+\delta_{\underline{\mathcal{B}}_{\hat{\mathbb{I}}}^Y}(Y)+\delta_{\underline{\mathcal{B}}_{\hat{\mathbb{J}}}^S}(S),\forall (X,Y,S)$. Since $\underline{\mathcal{B}}_{\hat{\mathbb{I}}}^X,\,\underline{\mathcal{B}}_{\hat{\mathbb{I}}}^Y,\underline{\mathcal{B}}_{\hat{\mathbb{J}}}^S$ are semialgebraic and $f$ is semialgebraic, by \cite[Section 4.3]{Attouch2010MOR}, $\hat{f}$ is a KL function with a {\L}ojasiewicz exponent $\alpha\in[0,1)$. Combining with \cite[Theorem 2.9]{Attouch2013} and \cite[Theorem 2]{Attouch2009}, it is not hard to prove the required statement.
\end{proof}

\begin{remark}
	There are some common smooth and semialgebraic functions, such as the squared loss function $f(Z,S)=\sum_{(i,j)\in\varOmega}(Z+S-M^o)_{ij}^2$ and the Huber loss function $f(Z,S)=\sum_{(i,j)\in\varOmega}h\big((Z+S-M^o)_{ij}\big)$, in which $h$ is the Huber function in \cite{huber1973robust} and $\varOmega\subseteq[m]\times[n]$ is an index set. Huber loss function is a popular loss function used in robust regression when dealing with outliers in the data.
\end{remark}

\begin{remark}
	In view of Theorem \ref{JA-PG-conv-thm} (i),
	\begin{itemize}
		\item for the low-rank inducing regularizer $\text{nnzc}(X)+\text{nnzc}(Y)$, only the nonzero columns of $X^k$ and $Y^k$ as in $\mathbb{I}_k$ need to be updated;
		\item for the sparse regularizer $\|S\|_0$, only the nonzero entries of $S^k$ as in $\mathbb{J}_{S^k}$ need to be updated after $\tilde{K}$ iterations.
	\end{itemize}
	Then the columns of $X^k$ and $Y^k$ and the entries of $S^k$ will no longer change once they are updated to zero after finite iterations. Based on the consistency of $X^k$ and $Y^k$, AJA-PG is actually solving problem (\ref{s_l}) with decreasing rank correspondingly.
\end{remark}
\begin{remark}
	Consider problem (\ref{s_l}) with ball constraints replaced by the box constraints $\mathcal{B}^X:=[\bm{\tau^1},\bm{\tau^2}]\subseteq\mathbb{R}^{m\times d}$ and $\mathcal{B}^Y:=[\bm{\pi^1},\bm{\pi^2}]\subseteq\mathbb{R}^{n\times d}$. If $\tau\geq\|\mathcal{B}^Z\|_{\max}^{1/2}$ is changed to $\mathcal{B}^X\times \mathcal{B}^Y\supseteq\|\mathcal{B}^Z\|_{\max}^{1/2}[-\bm{1},\bm{1}]$, then we can also get the relations among problems (\ref{m_s_l}), (\ref{s_l}) and (\ref{m_s_l_b}) in Proposition \ref{gl_re_c}. Note that if $\min_{i,j,k,l\in[d]}\{\bm\tau^1_i,\bm\tau^2_j,$ $\bm\pi^1_k,\bm\pi^2_l\}$ $\geq(\|\mathcal{B}^Z\|_{\max}/d)^{1/2}$, then $\mathcal{B}^X\times \mathcal{B}^Y\supseteq\|\mathcal{B}^Z\|_{\max}^{1/2}[-\bm{1},\bm{1}]$. Moreover, problem (\ref{s_l}) with such box constraints, also satisfies the corresponding theoretical results with $\|\cdot\|_2$ replaced by $\|\cdot\|_1$ in Sections \ref{section 2}, \ref{section 3} and \ref{section 4} without requiring the condition $r<\tau$ in Assumptions \ref{assumpars} and \ref{furpar}.
\end{remark}

\section{Numerical experiments}\label{section 5}
In this section, we test the performance of the AJA-PG algorithm for problem (\ref{s_l}) on synthetic and real datasets. Throughout this section, all computational results are obtained by running MATLAB 2022a on a MacBook Pro (3.2 GHz, 16 GB RAM). The running time (in seconds) of the AJA-PG algorithm is measured in the same manner as in \cite{Li2024,Pan2022factor}, including the time spent on computing the initial point. The numerical results reported are the average of the data generated by the algorithms involved in ten random trials.

We use the AJA-PG algorithm to solve relaxation problem (\ref{s_l_r}) with $\theta_1(t)=t,\forall t\in\mathbb{R}_+$, for the problem (\ref{s_l}), where $f(Z,S):=\frac{1}{2}\|\mathcal{P}_{\varOmega}(Z+S-M^o)\|_F^2,\forall Z\in\mathbb{R}^{m\times n}$ with the given matrix $M^o\in\mathbb{R}^{m\times n}$, and the index set $\varOmega\subseteq[m]\times [n]$ of observed entries. Considering the regularization parameter suggested in \cite{candes2011robust,zhang2015exact}, we use $\lambda=\sqrt{\max\{m,n\}}\,\beta/2$ if $\beta\neq0$ in the numerical experiments. Throughout this section, AJA-PG$_\text{ball}$ and AJA-PG$_\text{box}$ algorithms correspond to the AJA-PG algorithm for problem (\ref{s_l}) with ball and box constraints, respectively.

We denote the largest absolute value of all components of all matrices in $\mathcal{B}^Z\times \mathcal{B}^S$ from problem (\ref{m_s_l_b}) as $a_1$ and that of the matrix $M^o$ as $a_2$. Depending on the theoretical needs, we set $\Delta:=a_1\sqrt{mn}$, $L_X:=n(a_1+a_2)\sqrt{m\Delta}$, $L_Y:=m(a_1+a_2)\sqrt{n\Delta}$, $L_S:=a_1+a_2$, $r:=0.99\lambda/\max\{L_X,L_Y\}$ and $s:=0.99\beta/L_S$. Let ${r}_k=\sqrt{2\lambda},\forall k\leq k_{\max}$ and $\max\big\{\frac{\sqrt{2\lambda}}{k-k_{\max}},r\big\}$, otherwise, and ${s}_k=\sqrt{2\beta},\forall k\leq k_{\max}$ and $\max\big\{\frac{\sqrt{2\beta}}{k-k_{\max}},s\big\}$, otherwise. Set the initial point same as in \cite{Li2024}. Set the stopping criterion in subsection \ref{num_exp_1} for AJA-PG as $\frac{|F(X^k,Y^k,S^k)-F_0(X^{k-1},Y^{k-1},S^{k-1})|}{\max\{1,F(X^k,Y^k,S^k)\}}\leq 10^{-5}$ or the iteration number $k=500$, and the stopping criterion in subsection \ref{num_exp_3} same as in \cite{Li2024}.

%Set the stopping criterion in subsections \ref{num_exp_1} and subsection \ref{num_exp_3} same as in \cite{Li2024}, and the stopping criterion in subsections \ref{num_exp_2} for AJA-PG as $\frac{|F(X^k,Y^k,S^k)-F_0(X^{k-1},Y^{k-1},S^{k-1})|}{\max\{1,F(X^k,Y^k,S^k)\}}\leq 10^{-5}$ or the iteration number $k=500$. Set the initial point same as in \cite{Li2024}.

%In AJA-PG, denote the initial point as $(X_0,Y_0,S_0)$ and set $(X^0,Y^0)$ be only related to the observation $\mathcal{P}_{\varOmega}(M^o)$ as follows. By the singular value decomposition of $\mathcal{P}_{\varOmega}(M^o)$, we obtain $\mathcal{P}_{\varOmega}(M^o)=U\Sigma V^{\mathbb{T}}$ and then let $X^0=[\sqrt{\sigma_1}U_1,\ldots,\sqrt{\sigma_d}U_d]$ and $Y^0=[\sqrt{\sigma_1}V_1,\ldots,\sqrt{\sigma_d}V_d]$, where $\sigma_i$ is the $i$-th diagonal element of $\Sigma$. Based on the observed data, we use the rank estimation procedure in \cite{shang2018bilinear} to estimate the column number $d$ in $\mathcal{B}^X$ and $\mathcal{B}^Y$. 
\subsection{Synthetic data}\label{num_exp_1}
In this part, we consider problem (\ref{s_l}) with $\beta=1$, and use the same data setup\footnote{\url{https://www.dropbox.com/s/npyc2t5zkjlb7tt/Code BFMNM.zip?dl=0}} as in \cite{shang2018bilinear}. The low-rank matrix $Z^l$ of rank $r$ is generated as $PQ^\mathbb{T}$ , where $P\in\mathbb{R}^{n\times r}$ and $Q\in\mathbb{R}^{n\times r}$ are independent matrices whose elements are independent and identically distributed random variables sampled from the standard Gaussian distribution. The sparse matrix $S^u\in\mathbb{R}^{n\times n}$ is generated by the procedure that its support is chosen uniformly at random and its non-zero entries are independent and identically distributed random variables sampled uniformly in the interval $[-5,5]$. The input matrix is $M^o=Z^l+S^u+R^u$, where the Gaussian noise is $R^u=nf\times \textrm{randn}(n)$ and $nf\geq0$ is the noise factor. We compared the AJA-PG with the D-N and F-N algorithms in \cite{shang2018bilinear}, WNNM algorithm in \cite{gu2017weighted} and PSVT algorithm in \cite{OTH2016} under the same setting as in \cite{shang2018bilinear}. For quantitative evaluation, we measured the performance of low-rank component recovery by the RSE and running time. The comparison results under the same settings, as well as with varying sparsity and rank, are shown in Table \ref{table1}. Both AJA-PG$_\text{ball}$ and AJA-PG$_\text{box}$ algorithms have better numerical performance and the AJA-PG$_\text{box}$ algorithm takes the shortest time.
\begin{table}[htbp]
	\centering
	\caption{Average RSE and running time for different methods on the problems with $r^\star=5$, $sr=0.9$ and five square matrix dimensions}\label{table1}
	\resizebox{0.6\textwidth}{!}{\begin{tabular}{llcccccc}
			\toprule
			\multirow{2}{*}{Metric} & \multirow{2}{*}{Size ($n$)} & \multicolumn{6}{c}{Method} \\
			\cmidrule(lr){3-8}
			& & PSVT & WNNM & D-N & F-N & AJA-PG$_\text{ball}$ & AJA-PG$_\text{box}$  \\
			\midrule
			
			\multirow{5}{*}{RSE}
			& 500 & 0.130   &0.046   &0.045   &0.043   &\textbf{0.035}  &\textbf{0.035}   \\
			& 1000 & 0.128   &0.034   &0.032  &0.030  &\textbf{0.025}  &\textbf{0.025}  \\
			& 1500 & 0.128   &0.029   &0.027   &0.025   &\textbf{0.021}   &\textbf{0.021}   \\
			& 2000 & 0.127   &0.026   &0.023   &0.022   &\textbf{0.018}  &\textbf{0.018} \\
			& 2500 & 0.127   &0.024   &0.021  &0.020   &\textbf{0.016}  &\textbf{0.016} \\
			\midrule
			
			\multirow{5}{*}{Time}
			& 500 & 1.65 & 8.66 & 0.57 & 0.46 & 0.29 & \textbf{0.28}  \\
			& 1000 & 7.68   &37.63   &1.38   &1.22   &0.64  &\textbf{0.62}   \\
			& 1500 & 18.07 & 109.89 & 2.99 & 2.30 & 1.16 & \textbf{1.10}  \\
			& 2000 & 37.31 & 237.90 & 5.27 & 3.80 & 1.95 & \textbf{1.87}  \\
			& 2500 & 72.22   &500.96   &8.42   &7.59   &3.50  &\textbf{3.35} \\
			\bottomrule
	\end{tabular}}
\end{table}

\subsection{Real data}\label{num_exp_3}
In this part, we compare the performance of the AJA-PG algorithm and the first algorithm in \cite{Li2024} (denoted as AAPG) using real datasets for matrix completion under non-uniform sampling setting and different sampling ratios (SR) same as in \cite{Li2024}. We use some real datasets (Jester joke dataset and MovieLens dataset) to compare the numerical performance of the AJA-PG$_{\text{ball}}$ and AAPG algorithms for the considered problem in \cite{Li2024}. In Table \ref{real_data}, the datasets Jester-1, Jester-2 and Jester-3 are from the Jester joke dataset\footnote{\url{http://www.ieor.berkeley.edu/~goldberg/jester-data/}}, and the datasets Movie-100K and Movie-1M are from the MovieLens dataset\footnote{\url{http://www.grouplens.org/node/73}}. To measure the accuracy of computed solutions, same as in \cite{Li2024}, we adopt the normalized mean absolute error (NMAE). Under the same data and algorithmic settings as in \cite[subsection 6.2]{Li2024}, the comparison results under different sampling ratios (SR) are listed in Table \ref{real_data}. As the data dimension increases, the AJA-PG$_{\text{ball}}$ algorithm has a significant advantage in the running time, and achieves better solution accuracy in almost all tests.

\begin{table}[h]
	\centering
	\scriptsize
	\caption{Average NMAE, rank and running time of the AAPG and AJA-PG$_{\text{ball}}$ algorithms for Jester joke dataset and MovieLens dataset \label{real_data}}
	\resizebox{0.7\textwidth}{!}{\begin{tabular}{ccccccccc}
			\hline
			Dataset &($m$,\,$n$)&\!\! SR & \multicolumn{3}{l}{\qquad\qquad AAPG}
			&\multicolumn{3}{l}{\ \ \  \qquad\qquad AJA-PG$_\text{ball}$}\\
			\cmidrule(lr){4-6} \cmidrule(lr){7-9} 
			& & &\!\!\!\! NMAE&\!\!\! rank \!\!\!\! &time &  NMAE \!\!\!& rank \!\!\!& time\\
			\hline
			Jester-1
			&(1000,\,100)
			&0.15 
			& 0.1925 & 1.6 & \textbf{0.21}
			& \textbf{0.1899} & 1.8 & \textbf{0.21}\\
			& &0.25 
			& 0.1804 & 1.8 & \textbf{0.11}
			& \textbf{0.1802} & 1.8 & \textbf{0.11}\\
			\cmidrule(lr){1-9}
			Jester-2
			&(2000,\,100)
			&0.15 
			& 0.1902 & 1.8 & 0.39
			& \textbf{0.1894} & 1.8 & \textbf{0.33}\\
			& &0.25 
			& 0.1802  & 1.7 & 0.17
			& \textbf{0.1800} & 1.7 & \textbf{0.16}\\
			\cmidrule(lr){1-9}
			Jester-3
			&(4000,\,100)
			&0.15 
			& 0.2267 & 2.3 & 0.67
			& \textbf{0.2232} & 2.2 & \textbf{0.52}\\
			& &0.25 
			& 0.2213  & 2.5 & 0.45
			& \textbf{0.2174} & 2.4 & \textbf{0.38}\\
			\cmidrule(lr){1-9}
			Movie-100K
			&(943,\,1682)
			&0.10 
			& 0.2202 & 2.0 & 1.39
			& \textbf{0.2199} & 2.0 & \textbf{1.02}\\
			& &0.15 
			& 0.2138 & 1.9  & 1.21
			& \textbf{0.2128} & 1.5 & \textbf{0.73}\\
			& &0.20 
			& 0.2075 & 1.7 & 1.18
			& \textbf{0.2074} & 1.9 & \textbf{0.93}\\
			& &0.25 
			& 0.2042 & 1.6  & 1.04
			& \textbf{0.2041} & 1.5 & \textbf{0.72}\\
			\cmidrule(lr){1-9}
			Movie-1M
			&(6040,\,3706)  
			&0.10 
			& 0.2040 & 1.4 & 24.6
			& \textbf{0.2038} & 1.5 & \textbf{19.3}\\
			& &0.15
			& 0.1953 & 1.5 & 31.3
			& \textbf{0.1948} & 1.5 & \textbf{19.7}\\
			& &0.20 
			& 0.1918 & 1.7 & 32.6
			& \textbf{0.1917} & 1.6 & \textbf{16.1}\\
			& &0.25 
			& \textbf{0.1872} & 1.8 & 31.3
			& 0.1882 & 1.7 & \textbf{17.0}\\
			\hline
	\end{tabular}}
\end{table}

\section{Conclusions}\label{section 6}
This work is the first to study the rank and $\ell_0$ regularized matrix optimization problems based on matrix factorization. In the context of RPCA, the low-rank structure provides a computational advantage for the factorization model. We established the equivalences between the rank and $\ell_0$ regularized model and its factorization model in the senses of global minimizers and stationary points, respectively. For the bound constrained factorization model, we constructed an equivalent nonconvex relaxation framework with general adaptability. This established a multi-equivalence in the sense of optimality conditions between the factorization problem and a class of continuous nonconvex relaxation problems. Moreover, we designed a joint alternating proximal gradient (JA-PG) algorithm to solve the general factorization problem, and showed its convergence to the stationary points for the rank and $\ell_0$ regularized problem. Further, we proposed an adaptive JA-PG (AJA-PG) algorithm to solve the bound constrained factorization problem and its relaxation problems. We proved the convergence of the AJA-PG algorithm to the stationary points with stronger optimality conditions for these problems. Finally, through numerical experiments, we demonstrated the effectiveness of the proposed models and algorithms for some popular problems.

\bibliographystyle{plain}
\bibliography{references}
\end{document}